 \documentclass[SecEq,CM,GP]{degruyter-crelle} 

\firstpage{}\volume{}\volumeyear{}\copyrightyear{}\doiyear{}\doi{}\received{XXX}\revised{XXX}

\title[Models of Jacobians]{Models of Jacobians of curves}

\author{David}{Holmes}{}{Leiden}
\author{Sam}{Molcho}{}{Z\"urich}
\author{Giulio}{Orecchia}{}{Geneva}
\author{Thibault}{Poiret}{}{Cambridge}

\contact{Mathematisch Instituut, Universiteit Leiden, Netherlands}{holmesdst@math.leidenuniv.nl}
\contact{Department of Mathematics, ETH Z\"urich, R\"amistrasse 101, 8092, Z\"urich, Switzerland}{samolcho@math.ethz.ch}
\contact{Institute of Mathematics, EPFL SB MATH, MA C2 647 (B\^atiment MA) , Station 8, CH-1015 Lausanne, Switzerland}{giulioorecchia@gmail.com}
\contact{Department of pure mathematics and mathematical sciences, Cambridge, United Kingdom}{tp528@cam.ac.uk}

\researchsupported{ 
G. O. was supported by the Centre Henri Lebesgue, program ANR-11-LABX-0020-01. D.~H. was partially supported by NWO grants 613.009.103 and VI.Vidi.193.006. S.M was supported ERC-2017-AdG-786580-MACI and ERC Consolidator Grant 770922 - BirNonArchGeom. T.P was partially supported by Universiteit Leiden and EPSRC grant EP/V051830/1}

\usepackage{mathtools, tikz-cd}

\usepackage[all]{xy}
\usepackage[all]{xypic}

\usepackage{pgf,tikz}
\usetikzlibrary{matrix, calc, arrows}
\usepackage{tikz-cd} 
\usetikzlibrary{decorations.pathmorphing,shapes}

\makeatletter
\providecommand{\leftsquigarrow}{%
  \mathrel{\mathpalette\reflect@squig\relax}%
}
\newcommand{\reflect@squig}[2]{%
  \reflectbox{$\m@th#1\rightsquigarrow$}%
}
\makeatother

\newcommand{\abs}[1]{\lvert#1\rvert}

\usepackage{cleveref}
\crefformat{equation}{(#2#1#3)}
\let\oref\ref
\AtBeginDocument{\renewcommand{\ref}[1]{\cref{#1}}}
\numberwithin{equation}{subsection}

\def\sheafhom{\mathcal{H}om}
\def\sheafisom{\mathcal{I}som}

\def\sheafaut{\mathcal{A}ut}

\newcommand{\lra}{\longrightarrow}
\newcommand{\hra}{\hookrightarrow}
\newcommand{\sub}{\subseteq}

\newcommand{\iso}{\stackrel{\sim}{\lra}}

\newcommand{\cat}[1]{\bd{#1}}
\newcommand{\Sch}[1]{\cat{Sch}/#1}
\newcommand{\LSch}[1]{\cat{LSch}/#1}
\newcommand{\Schet}[1]{(\cat{Sch}/#1)_\et}
\newcommand{\LSchet}[1]{(\cat{LSch}/#1)_\et}
\newcommand{\shSchet}[1]{\on{Sh}(\cat{Sch}/#1)_\et}
\newcommand{\shLSchet}[1]{\on{Sh}(\cat{LSch}/#1)_\et}
\newcommand{\abSchet}[1]{\on{Ab}(\cat{Sch}/#1)_\et}
\newcommand{\abLSchet}[1]{\on{Ab}(\cat{LSch}/#1)_\et}
\newcommand{\sh}[1]{\on{Sh}(#1)}
\newcommand{\ab}[1]{\on{Ab}(#1)}

\renewcommand{\on}[1]{\operatorname{#1}}
\newcommand{\bb}[1]{{\mathbb{#1}}}

\newcommand{\ca}[1]{{\mathcal{#1}}}
\newcommand{\bd}[1]{{\mathbf{#1}}}

\newcommand{\ul}[1]{{\underline{#1}}}

\DeclareMathOperator*{\colim}{co{\lim}}
\DeclareMathOperator{\Ima}{Im}

\def\:{\colon}
\def\.{,\dots,}
\def\into{\hookrightarrow}
\newcommand{\oM}{{\overline{M}}}
\newcommand{\oN}{{\overline{N}}}

\def\ZZ{\mathbb Z}
\def\RR{\mathbb R}
\def\NN{\mathbb N}

\def\QQ{\mathbb Q}

\def\et{\mathrm{\acute{e}t}}

\def\Tropic{\operatorname{TroPic}}
\def\Logpic{\operatorname{LogPic}}
\def\sLPic{\operatorname{sLPic}}
\def\sTPic{\operatorname{sTPic}}
\def\sPic{\operatorname{sPic}}
\def\strLogPic{\operatorname{sLPic}}
\def\strTroPic{\operatorname{sTPic}}
\def\Pic{\operatorname{Pic}}
\def\Irr{\operatorname{Irr}}

\def\o#1{\overline{#1}}
\def\gp{\textrm{gp}}
\def\c#1{\mathcal{#1}}

\newcommand{\Sh}{\operatorname{Sh}}
\newcommand{\PSh}{\operatorname{PSh}}

\newcommand{\Hom}{\operatorname{Hom}}
\newcommand{\Spec}{\operatorname{Spec}}
\newcommand{\Proj}{\operatorname{Proj}}

\theoremstyle{definition}
\newtheorem{definition}{Definition}[section]

\newtheorem{property}[definition]{Property}
\newtheorem{fact}[definition]{Fact}

\theoremstyle{plain}
\newtheorem{conjecture}[definition]{Conjecture}

\newtheorem{proposition}[definition]{Proposition}
\newtheorem{lemma}[definition]{Lemma}
\newtheorem{theorem}[definition]{Theorem}
\newtheorem{corollary}[definition]{Corollary}

\theoremstyle{remark}

\newtheorem{remark}[definition]{Remark}

\newtheorem{example}[definition]{Example}

\renewcommand{\phi}{\varphi} 

\newcommand{\Dcomment}[1]{{\color{blue}D: #1}}
\newcommand{\Tcomment}[1]{{\color{red}T: #1}}
\newcommand{\Gcomment}[1]{{\color{teal}G: #1}}
\newcommand{\Scomment}[1]{{\color{purple}S: #1}}

\newcommand{\ZE}{\ca Z^\ca E}
\newcommand{\ZV}{\ca Z^\ca V}

\begin{document}

\begin{abstract}
We show that the Jacobians of prestable curves over toroidal varieties always admit N\'eron models. These models are rarely quasi-compact or separated, but we also give a complete classification of quasi-compact separated group-models of such Jacobians. In particular we show the existence of a maximal quasi-compact separated group model, which we call the saturated model, and has the extension property for all torsion sections. The N\'eron model and the saturated model coincide over a Dedekind base, so the saturated model gives an alternative generalisation of the classical notion of N\'eron models to higher-dimensional bases; in the general case we give necessary and sufficient conditions for the N\'eron model and saturated model to coincide. The key result, from which most others descend, is that the logarithmic Jacobian of \cite{Molcho2018The-logarithmic} is a log N\'eron model of the Jacobian.
\end{abstract}


\section{Introduction}\label{sec:intro}

\subsection{N\'eron models}
Let $X \rightarrow S$ be a prestable curve\footnote{This means that $X/S$ is proper, flat, finitely presented, and the geometric fibers are reduced and connected of pure dimension 1 and with at worst ordinary double point singularities; for example, a stable curve. } over a scheme $S$, smooth over a schematically dense open subscheme $U \subset S$. The Jacobian $\Pic^0_{X_U}$ is then an abelian scheme over $U$ which (in general) admits no extension to an abelian scheme over all of $S$. N\'eron suggested that one should look instead for a \emph{N\'eron model} of $\Pic^0_{X_U}$; a smooth algebraic space $N$ over $S$, such that $N \times_S U = \Pic^0_{X_U}$, satisfying the \emph{N\'eron mapping property:}

\emph{for any smooth map $T \rightarrow S$ of schemes, the natural restriction map $N(T) \to \Pic^0(T\times_S U)$ is a bijection.\footnote{The classical definition over Dedekind schemes and in \cite{Holmes2014Neron-models-an} requires N\'eron models to be separated. When the base is a Dedekind scheme and the generic fiber is a group scheme this is automatic by \cite[Theorem 7.1.1]{Bosch1990Neron-models}}.}

N\'eron models are unique when they exist, and inherit a group structure extending that of the Jacobian. When $S$ is Dedekind the existence of a N\'eron model was proven by N\'eron and Raynaud \cite{Neron1964Modeles-minimau}, \cite{Raynaud1966Modeles-de-Nero}; in this case, a N\'eron model is automatically separated and quasi-compact.

When the base $S$ is higher dimensional, a separated, quasi-compact N\'eron model rarely exists. In \cite{Holmes2014Neron-models-an}, Holmes showed that over a regular base $S$, existence of such a model implies a delicate relation between the smoothing parameters of the nodes of $X$, which he called alignment. He also proved that alignment is sufficient if the total space $X$ is also regular. In \cite{Orecchia2018A-criterion-for}, Orecchia refines this into a necessary and sufficient condition when $X$ is smooth over the complement of a normal crossings divisor $D \to S$, but not necessarily regular. In \cite{Poiret}, Poiret constructs the N\'eron model assuming any smooth $S$-scheme is locally factorial (e.g. $S$ is regular), and shows that a more restrictive version of alignment, \emph{strict alignment}, is necessary and sufficient for it to be quasi-compact and separated.

The main result of this paper is that without the extra conditions of separatedness and quasi-compactness, a N\'eron model of the Jacobian exists whenever the base is a toroidal variety (or more generally a log regular scheme, see \ref{sec:log_regular}). 

\begin{theorem}[\ref{coro:NMP_strict_log}]\label{thm:intro_mainresult}
Let $S$ be a toroidal variety with $U\subset S$ the open complement of the boundary divisor, and let $X/S$ be prestable curve, smooth over $U$. Then a N\'eron model $N/S$ for $\Pic^0_{X_U}$ exists and is quasi-separated. 
\end{theorem}

The N\'eron model admits a modular interpretation coming from logarithmic geometry: if one endows $X$ and $S$ with suitable logarithmic structures, one obtains the notion of \textit{logarithmic line bundle} on $X$ in the sense of \cite{Molcho2018The-logarithmic}. The N\'eron model is then the algebraic space $\sLPic^0_{X/S}$ (on the category of schemes) representing the functor of log line bundles of degree zero; we call it \textit{strict logarithmic Jacobian}. We postpone further discussion of the log geometric side of the story to \ref{subs:into_log} of this introduction. If $X$ is regular then a more concrete description of the N\'eron model can be given as a quotient of the relative Picard functor of $X/S$, generalising the approach of Raynaud when $\dim S = 1$; see \ref{sec:etale_closure}.

\subsection{Separated, quasi-compact models of the Jacobian}
We have seen that N\'eron models always exist, but they are rarely separated or quasi-compact. A \emph{model} of $\Pic^0_{X_U/U}$ is an algebraic space $G/S$ which restricts to $\Pic^0_{X_U/U}$ over $U$. Our next result gives a complete classification of separated, quasi-compact group models of the Jacobian $\Pic^0_{X_U/U}$ in terms of subgroups of the tropical Jacobian. In the setting of \ref{thm:intro_mainresult} we define the \emph{strict tropical Jacobian} $\sTPic^0_{X/S}$ to be the quotient of the N\'eron model by its fiberwise-connected component of identity $\Pic^0_{X_U/U}$. The quotient $\sTPic^0_{X/S}$ is an \'etale group algebraic space over $S$, and is trivial over $U$. If $S$ is local Dedekind then the group $\sTPic^0_{X/S}(S)$ is exactly the classical component group of the N\'eron model. 

Suppose we are given $\Psi/S$ an \'etale group space, and $\Psi\to \sTPic^0_{X/S}$ a group homomorphism. Then one obtains by fiber product a smooth group space $\c{G}(\Psi):=\Psi\times_{\sTPic^0_{X/S}}\sLPic^0_{X/S}$. If $\Psi_U=0$, then the restriction $\c{G}(\Psi)_U$ is identified with the Jacobian $\Pic^0_{X_U/U}$. 



\begin{theorem}[\ref{prop:separated_model}, \ref{coro:equivalence_sep_qf}]\label{thm:intro_bijection} The map $\Psi \mapsto \c{G}(\Psi)$ induces a bijection of partially ordered sets from the set of quasi-finite open subgroups of $\sTPic^0_{X/S}$ to the set of smooth, separated, quasi-compact $S$-group models of the Jacobian $\Pic^0_{X_U/U}$.
\end{theorem}

As an application, we consider the case $S=\overline{\ca{M}}_{g,n}$ the moduli stack of stable curves, and we take $X/S$ the universal curve. One shows that the strict tropical Jacobian is torsion-free in this case; an immediate consequence of \ref{thm:intro_bijection} is:
\begin{corollary}[\ref{lemma:model_universal_jacobian}]
The universal Jacobian $\Pic^0_{X/\ca M_{g,n}}$ admits a unique smooth, separated group model over $\overline{\ca M}_{g,n}$, namely the generalized Jacobian $\Pic^0_{ X/\overline{\ca M}_{g,n}}$ parametrizing line bundles of multidegree $(0,0,\ldots,0)$.
\end{corollary}

In particular this shows the non-existence of a smooth separated group scheme (or space) $\ca G$ over $\overline{\ca M}_{g,n}$ whose every fiber $\ca G_s$ is isomorphic to the special fiber of the N\'eron model of the Jacobian of a regular 1-parameter smoothing of $X_s$\footnote{This means a flat morphism of schemes $Y \to T$ where $T$ is a trait and $Y/T$ has smooth generic fiber; together with an isomorphism between $X_s/s$ and the special fiber of $Y/T$.}, as such a space would be a model of $\Pic^0_{X/\ca M_{g,n}}$. A question about the existence of such ``universal N\'eron models of the Jacobian" over a compactification of $\ca M_g$ was asked by Chiodo in the introduction of \cite{Chiodo2015Neron-models-of}. Caporaso investigates in \cite{Caporaso2008Neron-models-an} the analogous problem where, instead of Jacobians, one tries to fit into a universal family the special fibers of the N\'eron models of the $\Pic^d_{Y^s/T^s}$, where $Y^s/T^s$ is a regular $1$-parameter smoothing of $X_s$.

%
\subsection{The saturated model}

We define the \emph{saturated model} of the Jacobian $\Pic^0_{X_U/U}$ to be the smooth separated quasi-compact group model which is maximal for the relation of inclusion. 
\begin{corollary}[\ref{coro:maximal_model}]
The saturated model exists, and can be constructed as $\c{G}(\Psi)$ for $\Psi$ the torsion subgroup of $\sTPic^0_{X/S}$. 
\end{corollary}

When $\dim S = 1$ the N\'eron model and the saturated model coincide, so that the saturated model may be viewed as an alternative generalisation of the N\'eron model to base schemes of higher dimension. The saturated model has the advantage of being separated and quasi-compact, but has a weaker extension property: 
\begin{corollary}
Let $x\colon U\to  \Pic^0_{X_U/U}$ be a torsion section. Then $x$ extends uniquely to a section of the saturated model. 
\end{corollary}
%

The saturated model admits a modular interpretation: it represents the ``saturated Jacobian'' functor $\sPic^{sat}$ of log line bundles of degree zero on $X$ which become line bundles after taking a suitable integer power. 


\subsection{The saturated model and the N\'eron model}
The papers \cite{Holmes2014Neron-models-an}, \cite{Orecchia2018A-criterion-for}, \cite{Poiret} present criteria for the existence of quasi-compact and separated models of $\Pic^0_{X_U/U}$ when the base is regular. When put together, our main results yield the following criterion, valid even over some non-regular bases:
\begin{theorem}[\ref{thm:alignment_conditions}]
Let $S$ be a toroidal variety with $U\subset S$ the open complement of the boundary divisor, and let $X/S$ be prestable curve, smooth over $U$ (or more generally, a log smooth curve over a log regular base).Then the following are equivalent:
\begin{enumerate}
\item The strict tropical Jacobian $\sTPic^0_{X/S}$ is quasi-finite over $S$;
\item The N\'eron model of $\Pic^0_{X_U}$ is separated over $S$;
\item The N\'eron model of $\Pic^0_{X_U}$ is quasi-compact and separated over $S$; 
\item The saturated model and the N\'eron model of $\Pic^0_{X_U}$ are equal. 
\end{enumerate} 
\end{theorem}

\subsection{Log geometric interpretation}\label{subs:into_log}

Although many of our results concern classical algebraic geometry, they become more natural in the context of logarithmic geometry. To explain the connection, suppose that $X \rightarrow S$ is a family of logarithmic curves (see \ref{def: logcurve}). In \cite{Molcho2018The-logarithmic}, following ideas of Illusie and Kato, the authors constructed the analogue of the Picard scheme in the category of logarithmic schemes, the logarithmic Picard group $\Logpic_{X/S}$. This is the sheaf of isomorphism classes of the stack which parametrizes the logarithmic line bundles alluded to above, that is, certain\footnote{The torsors must satisfy a condition called bounded monodromy -- see section \ref{sec:big_tropical} for details. } torsors under the associated group $M_X^\gp$ of the log structure. 

The logarithmic Picard group is a group, is log smooth and proper over $S$, and on the locus $U$ of $S$ where the log structure is trivial it coincides with the ordinary Picard group $\Pic_{X_U}$. Furthermore, logarithmic line bundles have a natural notion of degree, extending the notion of degree for ordinary line bundles, and $\Logpic_{X/S}$ splits into connected components according to degree. Thus, the \textit{logarithmic Jacobian} $\Logpic^0_{X/S}$ provides a ``best possible" extension of the Jacobian $\Pic^0_{X_U/U}$. The caveat is that the logarithmic Jacobian is a sheaf on the category of log schemes, not schemes, and it is in general not algebraic -- i.e., it is not representable by an algebraic space with a log structure. In fact, it is ``log algebraic", that is, it satisfies the analogous properties that algebraic spaces enjoy, but only in the category of log schemes; for instance, it has a logarithmically \'etale cover by a log scheme. See \ref{ex:tate_curve} for the case of the Tate curve.

Nevertheless, properness of $\Logpic^0_{X/S}$ suggests that it is close to a N\'eron model for $\Pic^0_{X_U/U}$. For example, in the simplest case when $S$ is a trait\footnote{The spectrum of a discrete valuation ring. }, the valuative criterion tells us that every line bundle $L$ on $X_U$ extends uniquely to a log line bundle on $X$. In fact, this ``limit bundle'' is simply the pushforward $j_*L^\times$ of the $\ca O^\times$-torsor associated to $L$ along the inclusion $j\colon X_U\into X$. Remarkably, the description of the limit goes through whenever $S$ is a log regular scheme, showing that the logarithmic Jacobian satisfies the logarithmic version of the N\'eron mapping property:  

\begin{theorem}[\ref{thm:Neron}] \label{theo:intro_NMP_logpic}\label{thm:intro_logNMP}
Let $S$ be a log regular scheme. Then $\Logpic^0_{X/S}$ satisfies the N\'eron mapping property for log smooth morphisms.
\end{theorem}

In the case where $S$ is Dedekind this answers positively for Jacobians of curves a question of Eriksson, Halle, and Nicaise in \cite{Eriksson2015Logarithmic-interpretation}, who asked for the existence of a log N\'eron model. 

To connect \ref{thm:intro_logNMP} with classical algebraic geometry, we have to bring the problem back from the category of log schemes to the category of schemes. There is a standard procedure to do so: the category $\cat{Sch}/S$ embeds into $\cat{LSch}/S$ by giving $T \to S$ its pullback (``strict") log structure, and we may thus restrict the functor $\Logpic^0_{X/S}$ to $\cat{Sch}/S$. We denote the resulting functor, the ``strict" log Jacobian, by $\sLPic^0_{X/S}$. It is an immediate consequence that the strict logarithmic Jacobian $\sLPic^0_{X/S}$ satisfies the classical N\'eron mapping property. The functors $\Logpic^0_{X/S}$ and $\sLPic^0_{X/S}$ are very different in nature: good properties of $\Logpic^0_{X/S}$ such as properness, or even quasi-compactness, are generally lost in passing to $\sLPic^0_{X/S}$. This is however compensated by the following positive result:

\begin{theorem}[\ref{thm:sPic_representable}]
Let $X/S$ be a vertical log curve. The functor $\sLPic^0_{X/S}$ is representable by quasi-separated, smooth algebraic space over $\ul S$. 

If $S$ is log regular (e.g. a toroidal variety with divisorial log structure), then $\sLPic^0_{X/S}$ is the N\'eron model of $\Pic^0_{X_U/U}$. 
\end{theorem}

The strict tropical Jacobian $\sTPic^0_{X/S}$, defined above as a quotient, also has a natural log geometric interpretation. The log Picard group $\Logpic_{X/S}$ has a ``tropicalization" $\Tropic_{X/S}$, an essentially combinatorial object which determines the features of $\Logpic_{X/S}$ which are not present in the Jacobian $\Pic^0_{X/S}$ of $X/S$. Restricting the tropical Jacobian $\Tropic^0_{X/S}$ -- that is, the degree $0$ part of $\Tropic_{X/S}$ -- to schemes by giving a scheme its pullback log structure, as before, produces $\sTPic^0_{X/S}$. The tropical Jacobian plays an important role in the theory of compactifications of the universal Jacobian. It was essentially shown in \cite{Kajiwara1993Logarithmic-c,IIKajiwara2008Logarithmic-a} that subdivisions of $\Tropic^0_{X/S}$ correspond to toroidal compactifications of $\Pic^0_{X/S}$. \Cref{thm:intro_bijection} provides a complementary view of the role of $\Tropic_{X/S}$: the quasi-finite open subgroups of its strict locus determine the quasi-compact, smooth, separated group models of $\Pic^0_{X/S}$.

\section{Background}\label{sec:background}

Here we collect for the convenience of the reader the necessary facts that we will use, especially from the paper \cite{Molcho2018The-logarithmic}. 

\subsection{Log schemes}
All our log schemes are fine and saturated. For a log scheme $S$ we denote by $\ul S$ the underlying scheme. We denote by $\LSch{S}$ the category of log schemes over $S$, and by $\LSchet{S}$ the (big) strict \'etale site over $S$; the small strict \'etale site is denoted $S_{\et}$. 

We write $M_S$ for the log structure of a log scheme $S$, $\o{M}_S$ for its characteristic monoid $M_S/\c{O}_S^*$ (these are sheaves on $S_{\et}$). For a map of log schemes $f:X \rightarrow S$ we denote by $\oM_{X/S}$ the relative characteristic monoid $M_X/ f^*M_S = \o{M}_X/f^{-1}\o{M}_S$.

Let $S$ be a log scheme; we define the logarithmic and tropical multiplicative groups on $S$ to be the sheaves of abelian groups on $\LSchet{S}$ given by
\begin{align*}\mathbb G_{m,S}^{log}\colon (\ul T,M_T) \to M^{\gp}_T(\ul T)\\
 \mathbb G_{m,S}^{trop}\colon (\ul T,M_T) \to \oM^{\gp}_T(\ul T). 
\end{align*} 
\subsection{Log curves}

\begin{definition}
\label{def: logcurve}
A log curve $X\to S$ is a proper, vertical, integral, log smooth morphism of log schemes with connected and reduced geometric fibers of pure dimension 1.
\end{definition}
The underlying morphism of schemes $\ul X\to \ul S$ is a prestable curve as in \cite[\href{https://stacks.math.columbia.edu/tag/0E6T}{Tag 0E6T}]{stacks-project}. Our definition is the same as that of \cite{Kato2000Log-smooth-defo} except that we have added the assumption that the morphism be vertical; this means that the characteristic sheaf $\o{M}_{X/S}$ is a sheaf of groups, or equivalently that it is supported exactly on the non-smooth locus $X^{nsm}$ of $X$ over $S$.

\subsection{Sites, constructibility and representability}\label{subsec:constructibility}
We leave for a moment the category of log schemes. For a scheme $X$, we write $X_{\et}$ for the small \'etale site and $\Schet{X}$ for the big \'etale site.

There is a morphisms of sites $\mathfrak i\colon \Schet{X} \to X_{\et}$ given by the inclusion of categories $X_{\et} \to \Schet{X}$. We have functors between the categories of sheaves
  \begin{align}\label{eq:i_*}
 \frak i_*\colon &   \shSchet{X} \to \Sh(X_\et); \frak i_*\ca F(U/X) = \ca F(U) \;\;\; \text{ and}\\
 \frak i^*\colon & \Sh(X_\et) \to  \shSchet{X}
 \end{align}
  defined by setting $\frak i^*\ca F(j\colon T \to X)$ to be $j^*\ca F(T)$.
 
 \begin{definition}
A sheaf $\ca F \in \Sh({\cat{Sch}/X)_{\et}}$ is \emph{locally constructible} if the natural map $\frak i^*\frak i_*\ca F \to \ca F$ is an isomorphism.
  \end{definition}
  
Representability reduces to local constructibility via the well-known following lemma:

\begin{lemma}[\cite{Artin1973theoremes} VII, 1.8]\label{cor:loc_constructible_implies_representable_big}
Let $S$ be a Noetherian scheme. A sheaf $\ca F$ on $(\cat{Sch}/S)_{\et}$ is locally constructible if and only if it is representable by a (quasi-separated) \'etale algebraic space over $S$. 
\end{lemma}
\begin{remark}\label{remark:loc_sep_qsep}
The original statement of \ref{cor:loc_constructible_implies_representable_big} in \cite{Artin1973theoremes} states the equivalence for \'etale locally separated algebraic spaces; however an \'etale algebraic space is automatically locally separated as its diagonal is an open immersion (\cite[\href{https://stacks.math.columbia.edu/tag/05W1}{Tag 05W1}]{stacks-project}). Moreover, an \'etale algebraic space over a locally noetherian base is automatically quasi-separated, as it is easily seen by combining \cite[\href{https://stacks.math.columbia.edu/tag/03KG}{Tag 03KG}]{stacks-project} and \cite[\href{https://stacks.math.columbia.edu/tag/01OX}{Tag 01OX}]{stacks-project}.
\end{remark}
\subsection{Functors between the categories of log schemes and schemes}\label{subs:functors}
Fix a log scheme $S=(\ul S,M_S)$ and consider the functors
\begin{align*}
\frak f\colon \cat{LSch}/S  \to \cat{Sch}/{\ul S}  ; & \;\;\; T  \mapsto \underline T \\
\frak s\colon \cat{Sch}/{\ul S}  \to \cat{LSch}/S ; &\;\;\;  (g\colon T \to \ul S)  \mapsto ( T,g^*M_S)
\end{align*}
The first forgets the log structure; the second endows an $\ul S$-scheme with the strict (pullback) log structure from $S$.

The functor $\frak f$ is the left adjoint of $\frak s$: for $X\in \cat{LSch}/S$ and $Y\in\cat{Sch}/{\ul S}$, we have
$$\Hom(\underline X,Y)=\Hom(X,(Y,g^*M_S))$$

We write $\shLSchet{S}$ (resp. $\shSchet{\ul S}$) for the category of sheaves on the strict \'etale site on $\cat{LSch}/S$ (resp. the \'etale site on $\cat{Sch}/{\ul S}$). The functors $\frak f$ and $\frak s$ give rise to pushforward functors $\frak f_*$ and $\frak s_*$ on the categories of sheaves:
\begin{align*}
\shSchet{\ul S} & \stackrel{\frak f_*}{\longrightarrow} \shLSchet{S} \\
\shLSchet{S} & \stackrel{\frak s_*}{\longrightarrow} \shSchet{\ul S}. 
 \end{align*}
The functor $\frak s$ takes \'etale coverings to strict \'etale coverings, and commutes with fibered products. Therefore, $\frak s_*$ admits a left adjoint $\frak s^*\colon \shSchet{\ul S} \to \shLSchet{S}$ by \cite[\href{https://stacks.math.columbia.edu/tag/00WX}{Tag 00WX}]{stacks-project}. 

Since the counit $\frak f\circ \frak s\to 1_{\cat{Sch}}$ is an isomorphism, $\frak s_*\frak f_*$ is isomorphic to the identity. One then obtains by adjunction a morphism of functors $\frak s^*\to \frak f_*$.
\begin{lemma}\label{lemma:s^*f_*}
The map of functors $\frak s^*\to \frak f_*$ is an isomorphism. In particular, $\frak f_*$ is exact.
\end{lemma}
\begin{proof}
Let $\c{F}$ be a presheaf on $\cat{Sch}/{\ul S}$, and $\c{G}$ a presheaf on $\cat{\LSch{S}}$. For any log scheme $X \to S$ with underlying scheme $\ul{X}$ we have a canonical factorization $X \to \mathfrak{s}\ul{X} \to S$. Thus, for any map of presheaves $\mathfrak{f}_*\c{F} \to \c{G}$ and $X \in \cat{\LSch{S}}$, the map $\mathfrak{f}_*\c{F}(X) := \c{F}(\ul{X}) \to \c{G}(X)$ factors as $\c{F}(\ul{X}) \to \c{G}(\mathfrak{s}\ul{X}):= \mathfrak{s}_*\c{G}(X) \to \c{G}(X)$, i.e. $\mathfrak{f}_*$ is left adjoint to $\mathfrak{s}_*$. 
\end{proof}
\begin{remark}\label{remark:exactness}\leavevmode
We write $\on{Ab}(-)$ for the category of sheaves of abelian groups on a site. Both $\frak s$ and $\frak f$ are continuous and cocontinuous functors (see \cite[\href{https://stacks.math.columbia.edu/tag/00XJ}{Tag 00XJ}]{stacks-project} and \cite[\href{https://stacks.math.columbia.edu/tag/00WV}{Tag 00WV}]{stacks-project}); it follows that the pushforwards
\begin{align*}\frak s_*\colon \abLSchet{S}\to \abSchet{\ul S}& \text{ and}\\
\frak f_*\colon \abSchet{\ul S}\to \abLSchet{S}&
\end{align*}
 are exact, by \cite[\href{https://stacks.math.columbia.edu/tag/04BD}{Tag 04BD}]{stacks-project}. 
\end{remark}

The tropical multiplicative group satisfies the following useful property (not shared by its logarithmic counterpart $\mathbb G_{m,S}^{log}$):
\begin{lemma}\label{lemma:strictG_m}
Consider the sheaf $\mathbb G_{m,S}^{trop}$ on $\LSchet{S}$. We have a canonical isomorphism
$\frak s_*\mathbb G_{m,S}^{trop}=\frak i^*\oM^{\gp}_{\ul S}$
of sheaves on $\Schet{\ul S}$, where $\frak i^*$ is the functor from \ref{eq:i_*}. In particular $\frak s_*\mathbb G_{m,S}^{trop}$ is locally constructible.
\end{lemma}
\begin{proof}
Let $g\colon T\to \ul S$ be a morphism of schemes. Then
$$\frak s_*\mathbb G_{m,S}^{trop}(T)=\mathbb G_{m,S}^{trop}(\frak sT)=g^{-1}\oM^{\gp}_S(T)=\frak i^*\oM^{\gp}_S(T).\qedhere$$
\end{proof}

\begin{corollary}
The sheaves $\frak s_*\mathbb G_{m,S}^{trop}$ and $\frak s_*\mathbb G_{m,S}^{log}$ are representable by quasi-separated group algebraic spaces, respectively \'etale and smooth.
\end{corollary}
\begin{proof}
The statement for $\frak s_*\mathbb G_{m,S}^{trop}$ follows by \ref{cor:loc_constructible_implies_representable_big}.  Then we conclude by exactness of the sequence
$$0\to \mathbb G_{m,S}\to \frak s_*\mathbb G_{m,S}^{log}\to \frak s_*\mathbb G_{m,S}^{trop}\to 0.\qedhere$$ 
\end{proof}

\section{Tropical notions on the big site}\label{sec:big_tropical}

In this section we develop the necessary tools to introduce the tropical and logarithmic Jacobian of \cite{Molcho2018The-logarithmic}. We carefully define the sheaf of lattices $\mathcal H_{1,X/S}$ of first Betti homologies of the dual graphs of the fibers. Then, we construct the tropical Jacobian as a global quotient of a tropical torus $\sheafhom(\mathcal H_{1,X/S},\mathbb G_{m,S}^{trop})^{\dagger}$ by the sheaf of lattices $\mathcal H_{1,X/S}$; this will facilitate the proof of the representability of the strict tropical Jacobian in \ref{sec:representability}. The original definition of Molcho and Wise is slightly different but we show it to be equivalent to ours.

\subsection{Tropicalization}\label{subs:tropicalization}

A \emph{graph} consists of finite sets $V$ of vertices and $H$ of half-edges, with an `attachment' map $r\colon H \to V$ from the half-edges to the vertices, and an `opposite end' involution $\iota\colon H \to H$ on the half-edges. To be consistent with our convention that log curves are vertical, we require this involution $\iota$ to have no fixed points; because our curves have connected geometric fibers we also require our graphs to be connected. An edge is an unordered pair of half-edges interchanged by the involution, and we denote the set of them by $E$. 

Let $\o M$ be a sharp monoid. A \emph{tropical curve metrized by $\oM$} is a graph $(V,H,r,i)$ together with a function $\ell\colon E \to \o M\setminus\{0\}$. 

Let $S$ be a geometric logarithmic point\footnote{A log scheme whose underlying scheme is the spectrum of an algebraically closed field. }, and let $X/S$ be a log curve. The associated tropical curve (\emph{tropicalisation}) $\mathfrak X$ of $X$ has as underlying graph the usual dual graph of $X$ (with a vertex for each irreducible component and a half-edge for each branch at each singular point). To define the labelling $\ell\colon H \to \o M_S(S)$ we recall from \cite{Kato2000Log-smooth-defo} that the stalk of the log structure at a singular point $x$ has characteristic monoid $\o{M}_{X,x} \cong \o M_S(S) \oplus_{\NN} \NN^2$, where the coproduct is over the diagonal map $\NN \rightarrow \NN^2$, and a map $\NN \rightarrow \o M_S(S)$. If $e$ is the edge corresponding to $x$ then we set $\ell(e)$ to be the image of $1$ in $\o M_S(S)$; this is independent of the choice of presentation.


A map of monoids $\phi\colon \o M\to \o N$ determines a `contracted' tropical curve $\mathfrak X_{\phi}$, whose graph is obtained from the graph underlying $\mathfrak X$ by contracting all edges $e$ whose length $\ell(e)$ maps to $0$ via $\phi$, and with length of the remaining edges induced by $\phi$. Molcho and Wise define a tropical curve over an arbitrary log scheme $S$ as the data of a tropical curve for each geometric point of $S$, together with contraction maps between them compatible with geometric specialisations, but we will not use this notion.

If $S$ is a log scheme (but not necessarily a log point) and $X/S$ a log curve, the edge-labellings of the tropicalization of $X$ at various geometric points of $S$ vary nicely in families, as a consequence of the following proposition:

\begin{proposition}\label{proposition:special_log_structure}
Let $\pi \colon X \to S$ be a log curve and $\alpha \colon X^{nsm} \to S$ the structure morphism of the non-smooth locus of $\pi$. The category of cartesian squares
\begin{center}
\begin{tikzcd}
X \ar[r]\ar[d] & S \ar[d]\\
X' \ar[r] & S'
\end{tikzcd}
\end{center}
where $X' \to S'$ is a log curve and $S \to S'$ induces the identity on underlying schemes (with obvious morphisms) has a terminal object
\begin{center}
\begin{tikzcd}
X \ar[r]\ar[d] & S \ar[d]\\
X^\# \ar[r] & S^\#,
\end{tikzcd}
\end{center}
where the log structure $M_S^\#$ on $S^\#$ satisfies $\o M_S^\#=\alpha_*\NN$ (and $\NN$ denotes the constant sheaf with value $\NN$ on $X^{nsm}$).
\end{proposition}

\begin{proof}
This is \cite[Theorem 2.7]{Olsson2003Universal}, combined with the observation that $X/S$ is special in the sense of \cite[Definition 2.6]{Olsson2003Universal} if and only if $\oM_S=\alpha_*\NN$. 
\end{proof}

\subsection{Subdivisions}

Given a tropical curve $\mathfrak X=(V,H,r,i,\ell)$ with edges marked by a monoid $\oM$, and an edge $e=\{h_1,h_2\}$, one can define a new tropical curve $\mathfrak X'=(V',H',r',i',\ell')$ as follows: $V'$ is obtained from $V$ by adjoining a new vertex $v$, $H'$ is obtained by adding two half-edges $\gamma_1,\gamma_2$ with $r(\gamma_1)=r(\gamma_2)=v$. Then we set $i'(h_1)=\gamma_1$ and $i'(h_2)=\gamma_2$. Finally, we choose lengths $\ell'(\gamma_1)$ and $\ell'(\gamma_2)$ so that their sum is $\ell(h_1)$. On the remaining edges and vertices $r',i',\ell'$ are set to agree with $r,i,\ell$.
\begin{definition}
A tropical curve $\mathfrak X'$ constructed from $\frak X$ as in the paragraph above is called a \textit{basic subdivision} of $\mathfrak X$. A \textit{subdivision} of $\mathfrak X$ is a tropical curve obtained by composing a finite number of basic subdivisions. 
\end{definition}

We state a few simple facts regarding the behaviour of subdivisions with respect to contractions.

\begin{fact}
Let $\mathfrak X$ be a tropical curve metrized by a monoid $\oM$, $\phi\colon M\to N$ be a map of monoids, and $\mathfrak X^{\phi}$ be the contraction of $\mathfrak X$ induced by $\phi$ as in \ref{subs:tropicalization}. Let $\mathfrak Y$ be a subdivision of $\mathfrak X$. Then the contraction $\mathfrak Y^{\phi}$ is a subdivision of $\mathfrak X^{\phi}$.
\end{fact}
\begin{fact}
Let $S$ be a log scheme with log structure $M_S$ and $X/S$ a log curve. Let $(\bar \eta)\to (\bar s)$ be an \'etale specialization of geometric points, with induced morphism of characteristic monoids $\phi\colon\oM_{\bar s}\to \oM_{\bar \eta}$ (see \cite{cavalieri2017moduli}, Appendix A, for the notion of \'etale specialization). Let $\mathfrak X_{\bar s},\mathfrak X_{\bar \eta}$ be the tropicalizations of $X_{\bar s},X_{\bar \eta}$. Then $\mathfrak X_{\bar\eta}$ is the contraction $\mathfrak X^{\phi}_{\bar s}$ of $\mathfrak X_{\bar s}$ induced by $\phi$. 
\end{fact}
\begin{fact}\label{fact:modification_specialization}
Let $X \to S$ be a log curve, and $Y \to X$ a logarithmic modification such that $Y\to S$ is a log curve. For every log geometric point $t\to S$, the tropicalization of $Y_t$ is a subdivision of the tropicalization of $X_t$. Moreover, for an \'etale specialization $(\bar \eta)\to (\bar s)$ as above, the subdivision of $\mathfrak X_{\bar \eta}$ is $\mathfrak Y^{\phi}$, where $\mathfrak Y$ is the subdivision of $\mathfrak X_{\bar s}$ over $s$.
\end{fact}
\begin{lemma}[{\cite[2.4.3]{Molcho2018The-logarithmic}}]\label{lemma:subdivision_and_modification}
Let $\bar s\to S$ be a geometric point and $\mathfrak Y$ a subdivision of $\mathfrak X_{\bar s}$. 
Then there exists an \'etale neighbourhood $V$ of $\bar s$, a log curve $Y/V$ and a logarithmic modification $Y\to X\times_SV$ inducing the subdivision $\mathfrak Y\to \mathfrak X_{\bar s}$.
\end{lemma}

\subsection{The tropical Jacobian over a point}

Let $\frak X$ be a tropical curve over a monoid $\o M$. After choosing an orientation on the edges we have a boundary map $\bb Z^E \to \bb Z^V$, whose kernel is the first homology group $H_1(\frak X, \bb Z)$. Molcho and Wise define an intersection pairing
\begin{equation}
H_1(\frak X, \bb Z) \times H_1(\frak X, \bb Z) \to \o M^{\gp}, 
\end{equation}
and then define the Jacobian of $\frak X$ to be 
\begin{equation*}
\on{Hom}(H_1(\frak X, \bb Z), \o M^{\gp})^\dagger/H_1(\frak X, \bb Z), 
\end{equation*}
where the symbol $\dagger$ denotes the subgroup of elements of \emph{bounded monodromy}, see \cite[definition 3.5.5]{Molcho2018The-logarithmic}, or \ref{section:BM}.

For a log curve $X/S$, Molcho-Wise then define the tropical Jacobian to be the the data of the Jacobians of the tropicalisations of $X_s$ for every geometric point of $s$, together with their \'etale specialisation maps (see \cite[appendix A]{cavalieri2017moduli}). However, for our purposes it is important to upgrade the groups $\bb Z^E$ and $\bb Z^V$, and thereby $H_1(\frak X, \bb Z)$ and the tropical Jacobian, to sheaves of abelian groups on the big \'etale site of $S$. The definitions become slightly intricate, because a choice of orientation on the tropicalisation does not exist in families, and because locally constructible \emph{sheaves} of abelian groups are not isomorphic to their duals.

\subsection{The sheaf $\ZE$}
Let $\ul S$ be a scheme and $\pi\colon \ul X \to \ul S$ a prestable curve. Here we define a sheaf $\ZE$ on $\Schet{\ul S}$ which on each geometric point is isomorphic to the sheaf $\bb Z^E$ from above. Our sheaf is not of the form $\sheafhom(\ca E, \bb Z)$ for some sheaf $\ca E$ on $\Schet{\ul S}$, though it is locally of that form. Recall that $\frak i\colon \Schet{\ul X} \to {\ul X}_\et$ is the morphism of sites induced by the inclusion of categories ${\ul X}_\et \to \Schet{\ul X}$.
\begin{definition}
Equipping $X/S$ with the minimal log structure, we define the sheaves
\begin{equation}
\bb G_{m, X/S}^{trop} =  \frak i^*\oM^{\gp}_{X/S}.  \;\;\; \text{ and }\;\;\; \ZE = \pi_*\bb G_{m, X/S}^{trop}. 
\end{equation}
\end{definition}


\subsubsection{Branches, half-edges and orientations}\label{subs:branches_H}
\begin{definition}
We define $\c{E}$ (the \emph{sheaf of edges}) to be the sheaf on $\Schet{\ul S}$ represented by the non-smooth locus $\alpha\colon X^{nsm}\to \ul S$. We will use the notations $\c{E}$ and $X^{nsm}$ interchangeably, depending on context.
\end{definition}

 If we denote by $\{\star\}$ the final object in the category of sheaves of sets on $\Schet{X^{nsm}}$, then $\c{E}$ is $\alpha_!\{\star\}$.  Then \begin{equation}\label{Z^E=pushforward}\ZZ^{\c{E}}\coloneqq \sheafhom_{\ul S}(\c{E},\ZZ)=\sheafhom_{\ul S}(\alpha_!\{\star\},\ZZ)=\alpha_*\sheafhom_{X^{nsm}}(\{\star\},\alpha^*\ZZ)=\alpha_*\bb Z_{X^{nsm}}
\end{equation} but in general this is \emph{not} isomorphic to $\ZE$ due to the non-existence of global choices of orientations on the tropicalisations of the fibers; we expand on this in what follows. 

\begin{definition}\label{subs:branches}
The base change $X':=X\times_SX^{nsm}\to X^{nsm}$ along the finite unramified morphism $X^{nsm}\to S$ admits a natural section $a\colon X^{nsm}\to X\times_SX^{nsm}$ (a closed immersion). Let $Y$ denote the blowup of $X'$ along $a$, and define the \emph{scheme of branches}
\begin{equation}
X^{br}\coloneqq X^{nsm} \times_{X'} Y. 
\end{equation}
\end{definition}

 \begin{remark}\label{remark:split_branches_etale_locally}
 \begin{enumerate}
\item The projection $X^{br} \to X^{nsm}$ is finite \'etale of degree $2$; in fact, a $\ZZ/2\ZZ$-torsor;
\item The projection $X^{br}\to S$ is finite unramified; in particular, after replacing $S$ by an \'etale cover, both $X^{nsm}\to S$ and $X^{br}\to S$ are disjoint unions of closed immersions, so that locally on $S$ we have $X^{br}\cong X^{nsm}\sqcup X^{nsm}$ as schemes over $X$. 
\end{enumerate}
 \end{remark}

\begin{definition}
We say that $X\to S$ has \textit{split branches} if 
\begin{itemize}
\item
$X^{nsm}\to S$ is a disjoint union of closed immersions, and 
\item
$X^{br} \cong X^{nsm}\sqcup X^{nsm}$ over $X$. 
\end{itemize}
\end{definition}
\begin{remark}
A choice of section of the torsor $X^{br}\to X^{nsm}$ is equivalent to a choice of compatible orientations of the tropicalizations of the fibers of $X\to S$. 
\end{remark}

\begin{example}
Let $k$ be a field and $X$ be the prestable irreducible curve $\Proj k[x,y,z]/(y^2z-x^3-x^2z)$, with the involution $\sigma\colon X\to X$ given by $y\mapsto -y$. Consider the prestable curve $Y:= X\times_kX\to X$. Let $P$ be a $\bb Z/2\bb Z$-torsor on $X$, and consider the twisted form of $Y$ given by $Y^P:= Y\otimes_{\bb Z/2\bb Z}P \to X$, where $\bb Z/2\bb Z$ acts on $Y/X$ by $\sigma$. Then the scheme of branches of $Y^P/X$ is canonically isomorphic to the torsor $P\to X$. In particular, for $P$ whose class in $H^1(X_{\et},\bb Z/2\bb Z)=\bb Z/2\ZZ$ is non-zero, the curve $Y^P$ does not have split branches.
\end{example}

\begin{definition}
We define the \emph{sheaf of half-edges} $\ca H\in\Sh(\Schet{X^{nsm}})$ to be $X^{br}$. We will use the notations $\ca H$ and $X^{br}$ interchangeably. We will use the same notation for the pushforward (left adjoint to pullback) to $\Sh(\Schet{S})$.
\end{definition}

The sheaf $\ca H$ comes with a natural involution $\iota$ over $X^{nsm}$ and hence over $S$. Composition induces an involution $\iota$ on $\ZZ^{\c{H}}\coloneqq \sheafhom_{S}(\ca {H},\ZZ)$, and we find that 
\begin{lemma}\label{eq:inv_coinv}
\begin{equation}
\bb Z^\ca E = (\bb Z^\ca H)^\iota \;\;\; \text{ and }\;\;\; \ZE = (\bb Z^\ca H)^{-\iota}
\end{equation}
(respectively the invariants for $\iota$ and for $-\iota$). 
\end{lemma}

\begin{proof}
The equality $\bb Z^\ca E = (\bb Z^\ca H)^\iota$ is clear. Writing $\beta\colon X^{br} \to X$, the locally-free rank 1 $\bb Z$-module $\beta^*\oM^{\gp}_{X/S}$ has a \emph{canonical} generating section $m_0$ coming from generating sections of the ideal sheaves of the preimages of the node on the two branches. The equality $\beta = \beta \circ \iota$ induces an isomorphism $\iota^*\beta^*\oM^{\gp}_{X/S} \iso \beta^*\oM^{\gp}_{X/S}$, which takes $\iota^*m_0$ to $-m_0$. 

There is a natural isomorphism
\begin{equation}
\tau\colon \beta^*\oM^{\gp}_{X/S}  \iso \bb Z_{X^{br}} ; \;\; m_0  \mapsto 1.
\end{equation}
The involution $\iota$ induces an automorphism $\iota\colon \beta_*\bb Z_{X^{br}}\iso \beta_*\bb Z_{X^{br}}$. We then define a map
\begin{equation*}
 \oM^{\gp}_{X/S} \to \beta_*\bb Z_{X^{br}}
\end{equation*}
sending $m$ to $\tau(\beta^*m)$. Now $\iota(\tau(\beta^*m)) = - \tau(\beta^*m)$, so this map factors via the inclusion $(\beta_*\bb Z_{X^{br}})^{- \iota} \hookrightarrow \beta_* \bb Z_{X^{br}}$.
Locally on $X^{nsm}$ the induced map 
$
 \oM^{\gp}_{X/S} \to (\beta_*\bb Z_{X^{br}})^{-\iota}
$
 is a map of free rank-1 $\bb Z$-modules, and is easily checked to be an isomorphism, from which the second equality is immediate. 
\end{proof}
\begin{lemma}\label{lem:ZEisomZE}
The sheaves $\ZE$ and $\ZZ^{\c{E}}$  are isomorphic \'etale locally on $S$. 
\end{lemma}
\begin{proof}
Etale locally on $S$ the curve $X$ has split branches, so we conclude by \ref{eq:inv_coinv}. 
\end{proof}

\subsubsection{Representability}
\begin{lemma}\label{lemma:Z^E_representable}
The sheaf $\ZZ^{\c{E}}$ on $\Schet{S}$ is representable by a quasi-separated \'etale group algebraic space on $S$.
\end{lemma}
\begin{proof}
Working \'etale locally on $S$ we reduce to the case where $\c{E}\to \ul S$ is a finite disjoint union of closed immersions. We may then assume $i\colon \c{E}\to S$ is a closed immersion  with $\c E$ connected, and with complement $j\colon \ul U\to S$. The constant sheaf $\ZZ_{S}$ is representable by an \'etale algebraic space, and so is its open subgroup $j_!\bb Z_{ U}$. Hence the quotient $\sheafhom_{S}(\ca E,\bb Z)=i_*\bb Z_{\ca E}$ is representable by an \'etale algebraic space. For quasi-separatedness, we may reduce to $S$ locally noetherian because $\ca E$ is finitely presented, and then use \ref{remark:loc_sep_qsep}.
\end{proof}
Combining \ref{lem:ZEisomZE,lemma:Z^E_representable} yields
\begin{lemma}\label{lemma:pi_*representable}
The sheaf $\ZE$ is representable by a quasi-separated \'etale group algebraic space over $S$.
\end{lemma}

\subsection{The sheaf $\ZV$}

For any site $\Sigma$ we denote by $\ZZ[\;\;]\colon \sh{\Sigma}\to \ab{\Sigma}$ the functor taking a sheaf of sets $\ca F$ to the sheafification of the presheaf $\ZZ_{\PSh}[\ca F]$ of free abelian groups with basis $\ca F$. Note that $\ZZ[\;\;]$ is left adjoint to the forgetful functor $\ab{\Sigma}\to \sh{\Sigma}$; indeed, for any sheaves $\ca F$ of sets and $\ca G$ of abelian groups, we have
\begin{align*}
\on{Hom}_{\sh{\Sigma}}(\ca F,\ca G) & = \on{Hom}_{\PSh(\Sigma)}(\ca F,\ca G) \\
& = \on{Hom}_{\on{PAb}(\Sigma)}(\ZZ_{\PSh}[\ca F],\ca G) \\
& = \on{Hom}_{\ab{\Sigma}}(\ZZ[\ca F],\ca G). 
\end{align*}

\begin{definition}
Let $X/S$ be a prestable curve. We define the \emph{sheaf of relative irreducible components} $\Irr_{X/S}:=\pi_0(X^{sm}/S)$ by sending $T \to S$ to the set of $X_T$-isomorphism classes of open immersions $U\subset X_T$ such that
\begin{enumerate}
\item for every geometric point $\bar t$ of $T$, the fiber $U_{\bar t}$ is irreducible.
\item up to isomorphism, $U \to X_T$ is maximal among open immersions with property (i).
\end{enumerate}
\end{definition}

\begin{remark}\label{rem:romagny}
In \cite{Romagny2011Composantes-con}, Romagny studies relative irreducible components in more detail and generality. He shows that $\Irr_{X/S}$ is representable by a finitely presented \'etale algebraic space over $\ul S$. 
\end{remark}
\begin{definition}
We define the sheaf $\ZV \coloneqq \ZZ[\Irr_{X/S}]$ on $\Schet{S}$.
\end{definition}
\begin{remark}\label{remark:ZV_is_relative_Z^V}
The sheaf $\ZV$ is the N\'eron-Severi group of $\ul X/\ul S$. If $S$ is a geometric point then there is a natural isomorphism $\ZV \iso \bb Z^V$, motivating the notation.
\end{remark}
\begin{remark} The sheaf $\ZV$ should not be confused with the sheaf $\ZZ^{\Irr_{X/S}}=\Hom(\Irr_{X/S},\ZZ)$ on $\Schet{S}$. For example, if $S$ is a strictly henselian discrete valuation ring and $X/S$ a prestable curve, smooth over the generic point, and whose special fiber has two irreducible components, then the global sections of $\ZZ[\Irr_{X/S}]$ form a free abelian group of rank $2$, and those of $\ZZ^{\Irr_{X/S}}$ a free abelian group of rank $1$.
\end{remark}

\begin{lemma}
The sheaf $\ZV$ on $\Schet{S}$ is representable by a quasi-separated \'etale group algebraic space over $S$.
\end{lemma}
\begin{proof}
As $\Irr_{X/S}$ is of finite presentation, we may assume $S$ is locally Noetherian. By \ref{cor:loc_constructible_implies_representable_big} it suffices to prove that $\ZV$ is locally constructible. From \ref{rem:romagny} we know that $\Irr_{X/S}$ is locally constructible, in other words is of the form $\frak i^*J$ for some sheaf $J$ on the small \'etale site of $S$. Then
\begin{equation*}
    \bb Z[\Irr_{X/S}] = \bb Z[\frak i^*J] = \frak i^*\bb Z[J], 
\end{equation*}
where for the last equality we use that the functors $\frak i^*$ and $\ZZ[\;\;]$ commute, which follows in turn from the easy observation that their right adjoints ($\frak i_*$ and the forgetful functor from sheaves of abelian groups to sheaves of sets) commute. 
\end{proof}

%

\begin{lemma}\label{lemma:ZV_gen_by_local_sections}
Let $X \to S$ be a prestable curve. Then, the sheaves $\Irr:=\Irr_{X/S}$ and $\ZV=\ZZ[\Irr]$ on $\Schet{S}$ are generated by global sections \'etale-locally\footnote{For a sheaf $F$ on $S$, this means that every geometric point of $S$ admits an \'etale neighbourhoord $U$ over which the natural map from the constant sheaf $F(U)$ to the restriction $F|U$ is surjective. } on $S$. 
\end{lemma}

\begin{proof}
It suffices to prove $\Irr$ is generated by global sections \'etale-locally on $S$, which follows from \cite[Lemma 18]{Esteves2001Compactifying-t}.
\end{proof}

\subsection{The tropical boundary map}

If $\frak X$ is a tropical curve with a choice of orientation on the edges, there is a \emph{boundary map}
\begin{equation}\label{eq:boundary_map}
\delta\colon \bb Z^E \to \bb Z^V
\end{equation} 
sending an edge $E$ to its endpoint minus its startpoint. We will define an analogous map $\delta\colon \ZE \to \ZV$, independent of choices. 

\begin{remark}
As pointed out in \ref{subs:branches_H}, choices of orientations compatible with specialisation maps cannot in general be made in families, motivating our definition of $\ZE$ as $\pi_*\bb G_{m,X/S}^{trop}$ instead of as $\bb Z^\ca E$. 
\end{remark}

Since both $\ZE$ and $\ZV$ are locally constructible, it suffices to define $\delta$ on the small \'etale site of $S$. In what follows, we work exclusively on small \'etale sites, implicitly applying the functor $\frak i_*$ wherever necessary. Working locally on $S_{\et}$, we may assume that $X/S$ has split branches.


Write $X^{nsm}/S$ as a disjoint union of closed immersions $\bigsqcup_i Z_i \xrightarrow{\alpha_i} S$, and for each $i$ put $Z_i^{br}:=X^{br}\times_{X^{nsm}} Z_i$. Let $\ca H_i$ be $Z_i^{br}$, seen as a sheaf on $(Z_i)_\et$. We have $\ZE=\bigoplus_i \ZE_i$, where $\ZE_i$ is the subgroup of $(-\iota)$-invariants of $(\alpha_i)_*\ZZ^{\ca H_i}$. We will construct natural maps $\ZE_i \to \ZV$, and sum them. Pick some $i$ and put $Z:=Z_i$, $\alpha:=\alpha_i$. Here $\alpha$ is a closed immersion, and we use the notation from \cite[Proposition II.3.14]{Milne1980Etale-cohomolog} for pushing and pulling along $\alpha$ of sheaves on the small \'etale site. Denote by $j$ the open immersion $U:=S\backslash Z \to S$. By \cite[Proposition II.3.14]{Milne1980Etale-cohomolog}, $\alpha_*$ is left adjoint to the functor $\alpha^! \colon \ab{S_\et} \to \ab{Z_\et}$ taking a sheaf to (the pullback of) its subsheaf of sections supported on $Z$. Explicitly, for a sheaf $\ca F$ on $S$, $\alpha^!\ca F$ is the kernel of $\alpha^*\ca F \to \alpha^*j_*j^*\ca F$. We deduce
\begin{equation}\label{eq:adjunction_pullback_sections_with_fixed_support}
    \Hom_{Z_{\et}}((\ZZ^{\ca H_i})^{(-\iota)},\alpha^!\ZV)=\Hom_{S_{\et}}(\ZE_i,\ZV).
\end{equation}
There is a natural map of $S$-algebraic spaces 
\begin{equation}
b\colon X^{br} \to \Irr_{X/S}\into \ZV
\end{equation}
sending a point to the irreducible component of the normalisation of its fiber on which it lies. Writing $\beta\colon Z^{br} \to Z$, the map $b$ induces a global section of $\beta^*\alpha^*\ZV$, i.e. a morphism $\bb Z_{Z^{br}} \to \beta^*\alpha^*\ZV$. We have $\beta_*\bb Z_{Z^{br}}=\bb Z^{\ca H_i}$. Since $\beta$ is a disjoint union of isomorphisms, $\beta_*$ is simultaneously the left and right adjoint of $\beta^*$. By adjunction and after restricting to the $(-\iota)$-invariants, we get
a map $B \in \Hom_{Z_\et} ((\bb Z^{\ca H_i})^{-\iota}, \alpha^*\ZV)$. The image of $B$ consists of sections supported on $Z$, so $B$ factors through $\alpha^!\ZV$, and yields a map $\ZE_i \to \ZV$ by \ref{eq:adjunction_pullback_sections_with_fixed_support}. Summing over $i$, we get the desired map $\delta \colon \ZE \to \ZV$. The fact that $\delta$ does not depend on the chosen expression of $X^{nsm} \to S$ as a disjoint union of closed immersions is clear from the construction.

\begin{definition}
The \emph{tropical boundary map} is the map 
$
\delta\colon  \ZE \to  \ZV
$
constructed above.
\end{definition}

\begin{remark}\label{remark:boundary_map}
If $S$ is a log point we have a canonical isomorphism $\ZZ^{V}=\ZV$ (cf \ref{remark:ZV_is_relative_Z^V}). A choice of orientations of the edges of the tropicalization of $X/S$ provides an isomorphism $\ZZ^{\ca E}=\ZE$, which identifies $\delta$ with the boundary map of \ref{eq:boundary_map}.
\end{remark}

\begin{definition}
We denote by $\c{H}_{1,X/S}$ the kernel of $\delta$. It is a sheaf of abelian groups whose stalks are free and finitely generated. Its value at a geometric point $s$ of $S$ is isomorphic to the first homology group $H_1(\mathfrak X_s,\ZZ)$, cf. \ref{remark:boundary_map}. 
\end{definition}

\begin{remark}
Prompted by a request from an anonymous referee, we give an alternative interpretation of the map $\delta$, via the interpretation of $\ZV$ as the N\'eron-Severi group of $X/S$ (see \ref{remark:ZV_is_relative_Z^V}). For this remark, we work in the sites $\Schet{\ul{X}}$ and $\Schet{\ul{S}}$. Let $X^{\textup{st}}$ be $X$ with the log structure pulled back from $S$. We restrict sheaves naturally defined on $\LSchet{X^{st}}$, $\LSchet{X}$, $\LSchet{S}$ to $\Schet{\ul{X}}$, $\Schet{\ul{S}}$ without additional decorations, to avoid overcrowding the notation. Recalling that $\ZE = \pi_*\bb G^{trop}_{m , X/S}$, the cohomology of the short exact sequence 
\begin{equation}
0 \to \bb G_{m, X^\textup{st}}^{log} \to \bb G_{m, X}^{log} \to \bb G_{m, X/S}^{trop} \to 0
\end{equation}
gives a canonical map 
\begin{equation}
\ZE = \pi_*\bb G^{trop}_{m , X/S} \to R^1\pi_*\bb G_{m, X^{\textup{st}}}^{log}. 
\end{equation}
We also have an exact sequence 
\begin{equation}
0 \to \bb G_{m, X} \to \bb G^{log}_{m, X^{\textup{st}}} \to \pi^{-1}\bb G^{trop}_{m, S} \to 0, 
\end{equation}
which splits after base changing to a strict \'etale cover $S' \to S$. Choosing a splitting gives a map $\bb G^{log}_{m, X_{S'}^{\textup{st}}} \to \bb G_{m, X_{S'}}$, and taking cohomology gives a map 
\begin{equation}
R^1\pi_*\bb G^{log}_{m, X_{S'}^{\textup{st}}} \to R^1\pi_* \bb G_{m, X_{S'}} = \Pic_{X_{S'}/S'}. 
\end{equation}
This map depends on the choice of splitting, but we claim that the composite with the natural map to the N\'eron-Severi group $\Pic_{X_{S'}/S'}/\Pic_{X_{S'}/S'}^0$ does not depend on this choice. In particular, the composite descends to a canonical map
\begin{equation}
R^1\pi_*\bb G^{log}_{m, X^{\textup{st}}} \to \Pic_{X/S}/\Pic_{X/S}^0.
\end{equation}

The claim may be proven assuming $S'=S$. The difference between two choices of splittings is measured by a map 
\begin{equation}
F\colon \pi^{-1}\bb G^{trop}_{m, S} \to \bb G_{m, X}, 
\end{equation}
which by adjunction is the same as a map 
$
G\colon \bb G^{trop}_{m, S} \to \pi_*\bb G_{m, X} = \bb G_{m, S}
$
(the latter equality since $X/S$ is proper with reduced and connected geometric fibers). Then we have a map
\begin{equation}
R^1\pi_*\pi^{-1}G\colon  R^1\pi_*\pi^{-1}\bb G_{m, S}^{trop} \to  R^1\pi_*\pi^{-1}\bb G_{m, S}, 
\end{equation}
and composing with the canonical map $\pi^{-1}\bb G_{m, S} \to \bb G_{m, X}$ yields a map 
\begin{equation}
R^1\pi_*\pi^{-1}\bb G_{m, S}^{trop} \to R^1\pi_*\bb G_{m, X} = \Pic_{X/S}; 
\end{equation}
we must show that this map factors via $\Pic^0_{X/S}$. Writing $\beta\colon \pi^{-1}\bb G_{m,S} \to \bb G_{m,X}$ for the natural inclusion, we find that $F = \beta \circ \pi^{-1} \pi_* F$, 
and so 
\begin{equation}
R^1\pi_*F\colon R^1\pi_*\pi^{-1}\bb G_{m, S}^{trop} \to  R^1\pi_*\bb G_{m, X} = \Pic_{X/S}
\end{equation}
factors as
\begin{equation}
R^1\pi_*\pi^{-1}\bb G_{m, S}^{trop} \to R^1\pi_*\pi^{-1}\bb G_{m, S} \to \Pic_{X/S}, 
\end{equation}
and $R^1\pi_*\pi^{-1}\bb G_{m, S} \to \Pic_{X/S}$ factors through $\Pic^0_{X/S}$. Indeed, the fiber of $T:=R^1\pi_*\pi^{-1}\bb G_{m, S}$ over a geometric point $s \to S$ is the torus part of $\Pic^0_{X_s/s}$.

This shows the existence of a slight refinement of $\delta$; we defined it as a map from $\ZE$ to the N\'eron-Severi group $\ZV = \Pic_{X/S}/\Pic^0_{X/S}$, but it comes naturally from a map to the quotient $\Pic_{X/S}/T$ by the ``torus part of $\Pic_{X/S}$". We do not know if this lift has any application. 
\end{remark}

\subsection{The  monodromy pairing}
If $\frak X$ is a tropical curve over a monoid $\o M$, there is a map
$
E \times E \to \oM^{\gp}
$
sending $(e, f)$ to $\ell(e)$ if $e =  f$ and 0 otherwise. This induces a map $\bb Z^E \times \bb Z^E \to \o M^{\gp}$, which, after choice of an orientation on $E$, restricts to a pairing 
\begin{equation}
H_1(\mathfrak X,\ZZ) \times H_1(\mathfrak X,\ZZ) \to \o M^{\gp}
\end{equation}
called the \emph{monodromy pairing}. 
We will define an analogous map 
\begin{equation}\label{eq:pre_monodromy_pairing} \frak f_*\c{H}_{1,X/S}\times \frak f_*\c{H}_{1,X/S}\to \mathbb G^{trop}_{m,S}
\end{equation} where $\frak f_*$ is the functor $\abLSchet{S} \to \abSchet{\ul S}$ introduced in \ref{subs:functors}.

Let $X/S$ be a log curve. We write $\alpha\colon X^{nsm}\to \ul S$ and $\phi\colon X^{br} \to X^{nsm}$. On $\Schet{X^{nsm}}$, we have a sheaf of abelian groups $\phi_*\bb Z_{X^{br}}=\Hom_{X^{nsm}}(X^{br},\bb Z)$ which is locally free of rank $2$ and endowed with the involution $\iota$. We write $(\phi_*\bb Z_{X^{br}})^{-\iota}$ for the invariants under $-\iota$; this is a locally free rank $1$ sheaf of abelian groups, hence self-dual. Pushing forward the natural pairing $(\phi_*\bb Z_{X^{br}})^{-\iota}\times (\phi_*\bb Z_{X^{br}})^{-\iota} \to \bb Z$ via $\alpha$, we get a pairing
\begin{equation}\label{eq:prov_pairing}
\ZE\times\ZE\to \alpha_*\bb Z.
\end{equation}

Next, we compose \ref{eq:prov_pairing} with the natural map $\alpha_*\bb Z\to \frak i^*\oM^{\gp}_S$ of \ref{proposition:special_log_structure}, to obtain a pairing
\begin{equation}
\label{eqn:pre-monodromy-pairing}
\ZE\times \ZE \to \frak i^*\oM^{\gp}_S=\frak s_*\bb G_{m,S}^{trop},
\end{equation}
where the last equality is \ref{lemma:strictG_m}. 

\begin{definition}
The \textit{monodromy pairing} is the pairing 
$$\frak f_*\c{H}_{1,X/S}\times \frak f_*\c{H}_{1,X/S}\to \mathbb G^{trop}_{m,S}$$
obtained by adjunction from \ref{eqn:pre-monodromy-pairing} and restriction to $\c{H}_{1,X/S}$, and the fact that $\frak s^*=\frak f_*$ (\ref{lemma:s^*f_*})
\end{definition}

\subsection{The condition of bounded monodromy}\label{section:BM}
In order to define a tropical Jacobian for $X/S$, we would like to take the quotient of the tropical torus $\sheafhom(\frak f_*\c{H}_{1,X/S},\mathbb G_{m,S}^{trop})$ by the group of periods $\frak f_*\c{H}_{1,X/S}=\frak s^*\c{H}_{1,X/S}$, where the action is given by the monodromy pairing. However, the sheaf $\sheafhom(\frak f_*\c{H}_{1,X/S},\mathbb G_{m,S}^{trop})$ does not behave as we would like with respect to generization. Indeed, if $s,\eta$ are two geometric points of $S$ such that $\eta$ is a generization of $s$, a homomorphism $H_1(\mathfrak X_s)\to \oM^{\gp}_{S,s}$ does not induce a homomorphsim $H_1(\mathfrak X_{\eta})\to \oM^{\gp}_{S,\eta}$. As a consequence, formal elements of the sheaf do not integrate to elements over complete noetherian algebras. This has the consequence that the sheaf 
$$\frak s_*\sheafhom(\frak s^*\c{H}_{1,X/S},\mathbb G_{m,S}^{trop})=\sheafhom(\c{H}_{1,X/S},\frak s_*\mathbb G_{m,S}^{trop})$$ is not representable by an $\ul S$-algebraic space, since it fails to satisfy one of Artin's axioms (\cite{Artin1969Algebraization-}, Theorem 5.3, 2'). The authors in \cite{Molcho2018The-logarithmic} introduce the condition of \textit{bounded monodromy} which fixes this issue. We recall it now.

\begin{definition}
Let $\o M$ be a sharp fs monoid, and $a,b$ elements of $\o M^{\gp}$. We say $a$ is \emph{bounded by $b$} if there exist integers $n,m$ such that $nb\leq a\leq mb$ (for the partial order induced by $\o M$). The elements bounded by $b$ form a subgroup of $\o M^{\gp}$.
\end{definition}

\begin{definition}\label{def:bm}
Let $\mathfrak X$ be a tropical curve marked by a monoid $\oM$. A homomorphism $\phi\colon H_1(\mathfrak X)\to \oM^{\gp}$ is of \textit{bounded monodromy} if for every $\gamma\in H_1(\mathfrak X)$, $\phi(\gamma)$ is bounded by the length of $\gamma$, i.e. the self-intersection of $\gamma$ under the monodromy pairing.
\end{definition}

\begin{remark}\label{rem:length_additive_if_disjoint_supports}
The length $\ell$ is a map $H_1(\mathfrak{X}) \to \o M^{\gp}$ which factors through $\o M$. If two elements $\gamma,\gamma'$ of $H_1(\mathfrak{X})$ have disjoint support, then we have $\ell(\gamma+\gamma')=\ell(\gamma)+\ell(\gamma')$.
\end{remark}

This notion extends naturally to the relative setting:

\begin{definition}
Let $X \to S$ be a log curve. We define the morphism
\[
\ell \colon \frak f_*\c{H}_{1,X/S} \to \mathbb G_{m,S}^{trop}
\]
as the composition of the monodromy pairing with the diagonal $\frak f_*\c{H}_{1,X/S} \to \frak f_*\c{H}_{1,X/S}\times\frak f_*\c{H}_{1,X/S}.$
\end{definition}

\begin{definition}
Let $X \to S$ be a log curve. We call \emph{bounded monodromy subsheaf}, and write $\sheafhom(\frak f_*\c{H}_{1,X/S},\mathbb G_{m,S}^{trop})^{\dagger}$, for the open subgroup sheaf of $\sheafhom(\frak f_*\c{H}_{1,X/S},\mathbb G_{m,S}^{trop})$ obtained by sheafifying the sub-presheaf of morphisms $\phi\colon \frak f_*\c{H}_{1,X/S}(T) \to \o M_T^{\gp}$ such that for every $\gamma\in\frak f_*\c{H}_{1,X/S}(T)$, $\phi(\gamma)$ is bounded by $\ell(\gamma)$.
\end{definition}

\begin{remark}\
Suppose $S$ is a log point, and fix an orientation of the edges of the tropicalization $\mathfrak{X}$ of $X$ at $s$ (i.e. an isomorphism $\ZE=\ZZ^{\ca E}$). Then, we recover the usual notions of length and bounded monodromy on the tropical curve $\mathfrak{X}$.
\end{remark}

\begin{remark}
It will follow from \ref{section:eq_with_MW} that $\sheafhom(\frak f_*\c{H}_{1,X/S},\mathbb G_{m,S}^{trop})^{\dagger}$ is precisely the open subgroup sheaf of $\sheafhom(\frak f_*\c{H}_{1,X/S},\mathbb G_{m,S}^{trop})$ consisting of morphisms that have bounded monodromy at every strict geometric log point $s$ of $S$.
\end{remark}

One immediately checks that the morphism 
\(
\frak f_*\c{H}_{1,X/S}\to \sheafhom(\frak f_*\c{H}_{1,X/S},\mathbb G_{m,S}^{trop})
\)
induced by the monodromy pairing factors via the bounded monodromy subgroup. 

\subsection{The tropical Jacobian}
\begin{definition}\label{definition:tropical_jacobian}
The tropical Jacobian of $X/S$ is the sheaf on $(\cat{LSch}/S)_{\et}$
$$\Tropic^0_{X/S}=\sheafhom(\frak f_*\c{H}_{1,X/S},\mathbb G_{m,S}^{trop})^{\dagger}/\frak f_*\c{H}_{1,X/S}$$
\end{definition}

For $s$ a geometric point of $S$, a choice of an orientation of the graph $\mathfrak X_s$ induces an isomorphism
$$\Tropic^0_{X/S}(s)\cong \Hom(H_1(\mathfrak X_s),\oM^{\gp}_{S,s})^{\dagger}/H_1(\mathfrak X_s). $$

If $\frak X/\oM$ is a tropical curve, we write $\Tropic^0(\frak X/\oM)$ for the  group $\Hom(H_1(\mathfrak X_s),\oM^{\gp}_{S,s})^{\dagger}/H_1(\mathfrak X_s)$. 

As explained in \cite{Molcho2018The-logarithmic}, for $\eta$ a generization of $s$, any homomorphism of bounded monodromy $H_1(\mathfrak X_s)\to \oM^{\gp}_{S,s}$ induces a unique homomorphism (of bounded monodromy) $H_1(\mathfrak X_{\eta})\to \oM^{\gp}_{S,\eta}$; moreover an orientation of $\mathfrak X_s$ induces a unique orientation of its contraction $\mathfrak X_{\eta}$. There is therefore an induced generization map 
\begin{equation}\label{eqn:generization_map}
\Tropic^0_{X/S}(s)\to \Tropic^0_{X/S}(\eta)
\end{equation}

The original definition of $\Tropic^0_{X/S}$ in \cite{Molcho2018The-logarithmic} is slightly different than the one we gave and relies on the generization maps defined above. The rest of this section is devoted to verifying that the two definitions are actually equivalent. 

\subsection{Equivalence with the definition of Molcho-Wise}\label{section:eq_with_MW}

All our log schemes are fine and saturated, so in particular they have \'etale-local charts by finitely presented monoids. This allows for the following convenient definitions of finiteness conditions for log schemes.

\begin{definition}\label{definition:morphism_of_logsch_of_finite_type}
	We say that a log scheme is \emph{of finite type}, resp. \emph{locally of finite type}, resp. \emph{Noetherian}, resp. \emph{locally Noetherian} if its underlying scheme is.
	
	We say that a morphism of log schemes is \emph{locally of finite type}, resp. \emph{of finite type}, resp. \emph{locally of finite presentation}, resp. \emph{of finite presentation} if the underlying scheme map is.
\end{definition}

\begin{remark}\label{remark:log_curves_are_finitely_presented}
	Any log curve is of finite presentation. In particular, any log curve is the base change of a log curve over a log scheme locally of finite type.
\end{remark}

\begin{definition}\label{def:atomic}
Let $\pi \colon X\to S$ be a log curve. We say that $S$ is \textit{nuclear} (with respect to $\pi$) if:
\begin{itemize}
\item[1)] the stratification of $S$ induced by the log structure $\oM_S$ has only one closed stratum $Z$, $Z$ is connected, and every connected component of every stratum specializes to $Z$;
\item[2)] $\oM_S$ is generated by global sections (in particular there exists a surjection $\ZZ^{J} \to \frak i^*\oM_S^{\gp}$ where $J$ is a set);
\item[3)] $X/S$ has split branches; 
\item[4)] the sheaf $\ZV$ is generated by global sections.
\end{itemize}
We say that $S$ is \emph{pre-nuclear} if it satisfies 2), 3) and 4). We say that $S$ is a \emph{nuclear neighbourhood} of a geometric point $t$ of $S$ if $S$ is nuclear and $t$ maps to the closed stratum.
\end{definition}

\begin{remark}\label{remark:pre-nuclear_preserved_by_basechange}
	Conditions 2), 3) and 4) are stable under strict \'etale base change. In particular so is pre-nuclearity.
\end{remark}

\begin{remark}\label{remark:pre_nuclear_implies_bar_M_S_constant_on_strata}
	Suppose $S$ is locally Noetherian and satisfies condition 2), and let $Z$ be a stratum of $S$ for the stratification induced by $\oM_S$. Then the sheaves $\oM_S$ and $\oM_S^{\gp}$ are \'etale-locally constant on $Z$ by definition of the stratification, and they are generated by global sections, so they are constant on each connected component of $Z$.
\end{remark}

\begin{lemma}\label{lemma:cover_by_nuclear}
If $X/S$ is a log curve, then $S$ admits a strict \'etale cover by pre-nuclear schemes. If in addition $S$ is locally Noetherian, then any geometric point has a nuclear strict \'etale neighbourhood.
\end{lemma}
\begin{proof}
First we show that $S$ has a cover by pre-nuclear schemes. By \ref{remark:pre-nuclear_preserved_by_basechange}, it suffices to show that conditions 2), 3) and 4) individually hold locally on $S$ for the strict \'etale topology. We can assume that $S$ meets condition 3) by \ref{remark:split_branches_etale_locally}; condition 4) by \ref{lemma:ZV_gen_by_local_sections}; and condition 2) by the existence of \'etale local charts for log schemes.

Now, assume $S$ is locally Noetherian and let $t$ be a geometric point of $S$. We will show that $t$ has a nuclear strict \'etale neighbourhood. Shrinking, we may assume that $S$ is pre-nuclear and Noetherian. In particular, the connected components of strata (of the stratification induced by $\o M_S$) form a partition of $S$ into finitely many locally closed subschemes. If the closure $\o Y$ of a piece $Y$ does not meet $t$, take out $\o Y$ from $S$. The resulting $S$ is a nuclear neighbourhood of $t$.
\end{proof}

\begin{lemma}\label{lemma:generization_tropicalization}
Let $X/S$ be a log curve with $S$ locally Noetherian and pre-nuclear. Let $Z$ be a connected component of a stratum of $S$ (for the stratification induced by $\o M_S$). Then for any two geometric points $t,t'$ in $Z$, there is a canonical isomorphism $\frak X(t)=\frak X(t')$ of tropical curves over $\oM_S(Z)$.
\end{lemma}
\begin{proof}
The sheaf $\o M_S$ is constant on $Z$ by \ref{remark:pre_nuclear_implies_bar_M_S_constant_on_strata}. Since $Z$ is locally Noetherian, $t$ and $t'$ can be joined by a sequence of geometric points
$$t_0=t \leftsquigarrow \eta_0 \rightsquigarrow t_1 \leftsquigarrow \eta_2\rightsquigarrow \ldots \rightsquigarrow t_n=t'$$
where the squiggly arrows denote \'etale specialization, and all points land in $Z$. Every specialization $\eta_i\rightsquigarrow t_j$ induces an edge contraction $\mathfrak X_{t_j}\to \mathfrak X_{\eta_i}$, which is an isomorphism of tropical curves over $\o M_S(Z)$. The automorphism of $\mathfrak X_t$ induced by any such \'etale path from $t$ to itself must be trivial: it sends every vertex to itself since $\Irr_{X/S}$ is generated by global sections, and every half-edge to itself since $X/S$ has split branches. Therefore, the isomorphism $\mathfrak X_t \to \mathfrak X_{t'}$ is independent of the choice of \'etale path.
\end{proof}

\begin{corollary}\label{corollary:generization_tropic}
With the hypotheses and notations of \ref{lemma:generization_tropicalization}, the sheaf $\sheafhom(\frak f_*\mathcal H_{1,X_Z/Z},\bb G^{trop}_{m,Z})^\dagger$ is constant. In particular, $\Tropic^0_{X/S}$ is constant on $Z$ as well.
\end{corollary}

\begin{proof}
By base-change, we can reduce to the following claim: if $S$ is connected and $\o M_S$ constant, then for any strict geometric point $s \to S$, the restriction map
\[
\sheafhom(\frak f_*\mathcal H_{1,X/S},\bb G^{trop}_{m,S})^\dagger(S) \to \sheafhom(\frak f_*\mathcal H_{1,X/S},\bb G^{trop}_{m,S})^\dagger(s)
\]
is an isomorphism. Under the hypotheses of the claim, the sheaves $\mathcal Z^{\mathcal E}$, $\mathcal Z^{\mathcal V}$ and $\frak s_*\bb G^{trop}_{m,S}$ are constant on $\ul S$, so $\mathcal H_{1,X_S/S}$ is constant as well. In particular we have
\begin{equation}\label{eqn:restriction_monodromy_maps}
\sheafhom(\frak f_*\mathcal H_{1,X/S},\bb G^{trop}_{m,S})(S) = \sheafhom(\frak f_*\mathcal H_{1,X/S},\bb G^{trop}_{m,S})(s).
\end{equation}
Let $\phi$ be a global section of $\sheafhom(\frak f_*\mathcal H_{1,X/S},\bb G^{trop}_{m,S})$. If $\phi$ has bounded monodromy, then its image $\phi_s$ in $\sheafhom(\frak f_*\mathcal H_{1,X/S},\bb G^{trop}_{m,S})(s)$ also does. Conversely, suppose $\phi_s$ has bounded monodromy. Since $\mathfrak f_*\mathcal H_{1,X_S/S}(S)=\mathfrak f_*\mathcal H_{1,X_S/S}(s)\iso H_1(\mathfrak X_s)$, for every $\gamma\in\frak f_*\mathcal H_{1,X/S}(S)$, the image $\phi(\gamma)$ is bounded by the length of $\gamma$ in $\o M_S(S)=\o M_S(s)$. Hence, $\phi$ has bounded monodromy. Thus, the isomorphism \ref{eqn:restriction_monodromy_maps} respects the bounded monodromy subgroups, and the claim follows.
\end{proof}

\begin{lemma}\label{lemma:atomic_bm}
Let $X/S$ be a log curve with $S$ locally Noetherian and nuclear. Let $s$ be a geometric point of the closed stratum of $S$. Then the restriction map
$$\Tropic^0_{X/S}(S)\to \Tropic^0_{X/S}(s)$$
is an isomorphism.
\end{lemma}
 \begin{proof}
 
Since $\c{E}=X^{nsm}\to S$ is a disjoint union of closed immersions, and $s$ lies in the unique closed stratum, the map $\ZZ^{\c{E}}(S)\to \ZZ^{\c{E}}(s)$ is an isomorphism. Because $X/S$ has split branches, we can make a choice of isomorphism $\sigma\colon \ZZ^{\c{E}}\cong \ZE$, so the restriction map $\ZE(S)\to \ZE(s)$ is an isomorphism. Since additionally $\ZV(S)=\ZV(s)$, it follows that the restriction $\c{H}_{1,X/S}(S)=\c{H}_{1,X_s/s}(s)$ is an isomorphism.

Next, we show that the restriction map 
$$\sheafhom(\frak f_*\c{H}_{1,X/S},\mathbb G_{m,S}^{trop})^{\dagger}(S)\to \sheafhom(\frak f_*\c{H}_{1,X_s/s},\mathbb G_{m,s}^{trop})^{\dagger}(s)$$ is an isomorphism. By adjunction, the map above is the map  
$$\Hom(\c{H}_{1,X/S},\frak s_*\mathbb G^{trop}_{m,S})^{\dagger}\to \Hom(\c{H}_{1,X_s/s},\frak s_*\mathbb G_{m,s}^{trop})^{\dagger} $$
of Hom groups between sheaves on $\Schet{\ul S}$.
The two sheaves $\c{H}_{1,X/S}$ and $\frak s_*\mathbb G_{m,S}^{trop}=\frak i^*\o M^{\gp}_S$ are both locally constructible. To give a morphism $\c{H}_{1,X/S} \to  \frak s_*\mathbb G_{m,S}^{trop}$ it suffices therefore to give a morphism $\frak i_*\c{H}_{1,X/S}\to \oM^{\gp}_{S}$ of sheaves on the small \'etale site. We are reduced to proving that
\begin{equation}\label{eqn:restriction_bm}\Hom(\frak i_*\c{H}_{1,X/S},\oM^{\gp}_S)^{\dagger}\to \Hom(H_1(\mathfrak X_s),\oM^{\gp}_s)^{\dagger} \end{equation}
is an isomorphism.

To give a map $\phi\colon \frak i_*\c{H}_{1,X/S}\to \oM^{\gp}_S$, it suffices to give, for every stratum $W\subset S$, a map $\phi_W$ between the restrictions to $W$, subject to the following compatibility condition: for every two strata $W,W'$ such that $W'$ is contained in the closure of $W$, let $U=W\cup W'$ (a locally closed in $S$) and $h\colon W'\into U$, $j\colon W\into U$, respectively a closed and open immersion. Then the diagram
\begin{center}
\begin{tikzcd}
h^*\c{H}_{1,X_U/U} \ar[r, "\phi_{W'}"]\ar[d] & h^*\oM^{\gp}_{U} \ar[d]\\
h^*j_*j^*\c{H}_{1,X_U/U} \ar[r, "h^*j_*\phi_W"] & h^*j_*j^*\oM^{\gp}_{U}
\end{tikzcd}
\end{center}
should commute.

For every connected component $W$ of a stratum, we choose a point $\eta_W$. Remember that we have made a global choice of isomorphism $\ZE=\pi_*\mathbb G_{m,X/S}^{trop}\cong \ZZ^{\c{E}}$ and that $\c{H}_{1,X/S}$ is constant on connected components of strata. Thus $\c{H}_{1,X_W/W}$ is identified with the constant sheaf with value $H_1(\frak X_{\eta_W})$. Similarly, $\oM^{\gp}_W$ is the constant sheaf with value $\oM^{\gp}_{\eta_W}$, this time by condition 2) of \ref{def:atomic}.

Let then $t$ be a point of $S$ and $\eta$ a generization of $t$. We show that for every map of bounded monodromy $\alpha_t\colon H_1(\mathfrak X_t)\to \oM^{\gp}_t$, the solid diagram 

\begin{center}
\begin{tikzcd}
H_1(\mathfrak X_t) \ar[r, "\alpha"]\ar[d] & \oM^{\gp}_t \ar[d] \\
H_1(\mathfrak X_{\eta} )\ar[r, dashed] & \oM^{\gp}_{\eta}
\end{tikzcd}
\end{center}
 admits a unique dashed arrow $\alpha_{\eta}$ making the diagram commute.
The uniqueness is due to the surjectivity of the left vertical map. For the existence part, a dashed arrow exists if $\alpha$ sends the kernel $K$ of the left vertical map into the kernel $K'$ of the right vertical map, i.e. if any cycle $\gamma\in H_1(\mathfrak X_t)$ with length $\ell(\gamma)$ vanishing in $\oM^{\gp}_{\eta}$ maps to zero in $\oM^{\gp}_{\eta}$. This holds by the bounded monodromy condition, which completes the proof of existence. We write $(\eta\rightsquigarrow s)^*\alpha_t$ for the unique map $\alpha_{\eta}$ just constructed.

We omit the easy check that for every diagram of generizations 
\begin{center}
\begin{tikzcd}
\eta \arrow[r, rightsquigarrow] \arrow[d,squiggly] & \eta_1 \arrow[d,squiggly]\\
\eta_2 \arrow[r, squiggly] & t 
\end{tikzcd}
\end{center}
and map of bounded monodromy $\alpha_t\colon H_1(\mathfrak X_t)\to \oM^{\gp}_t$, the maps $(\eta\rightsquigarrow \eta_1)^*(\eta_1\rightsquigarrow t)^*\alpha_t$ and $(\eta\rightsquigarrow \eta_2)^*(\eta_2\rightsquigarrow t)^*\alpha_t$ coincide.

Call $Z$ the closed stratum of $S$. The above discussion, combined with condition 1) of the definition of nuclearity, shows that an arbitrary $\phi\in\Hom(H_1(\mathfrak X_s),\oM^{\gp}_s)^{\dagger}=\Hom(\frak i_*\c{H}_{1,X_Z/Z},\oM^{\gp}_Z)^{\dagger}$ fits in a unique compatible system of bounded monodromy maps $\frak i_*\c{H}_{1,X_W/W}\to\oM^{\gp}_W$ where $W$ ranges through the connected components of strata, i.e. $\phi$ comes uniquely from an element of $\Hom(\frak i_*\c{H}_{1,X/S},\oM^{\gp}_S)^{\dagger}$, which shows that $\sheafhom(\frak f_*\c{H}_{1,X/S},\mathbb G_{m,S}^{trop})^{\dagger}(S)\to \sheafhom(\frak f_*\c{H}_{1,X_s/s},\mathbb G_{m,s}^{trop})^{\dagger}(s)$ is an isomorphism.

Summarizing, we have shown that the composition 
\[
\begin{tikzcd}
    \sheafhom(\frak f_*\c{H}_{1,X/S},\mathbb G_{m,S}^{trop})^{\dagger}(S)/\frak f_*\c{H}_{1,X/S}(S) \ar[r] & \Tropic^0_{X/S}(S) \ar[r] & \Tropic^0_{X/S}(s)
\end{tikzcd}
\]
is an isomorphism. In particular, the restriction map $\Tropic^0_{X/S}(S) \to \Tropic^0_{X/S}(s)$ is surjective. We will show that it is also injective. Let $\lambda$ be an element of its kernel. It suffices to prove that $\lambda$ vanishes in the \'etale stalk $\Tropic^0_{X/S}(\Spec\ca O_{S,\eta}^{et}) = \Tropic^0_{X/S}(\eta)$ at any geometric point $\eta$ of $S$. Let $W$ be the stratum containing $\eta$. Over some generic point of $W$, there is a geometric point $\eta_W$ specializing to both $\eta$ and a point of $Z$. By \ref{corollary:generization_tropic} we know $\Tropic^0_{X/S}$ is constant on strata, so the restriction map
$$
\Tropic^0_{X/S}(S) \to \Tropic^0_{X/S}(\eta_W) = \Tropic^0_{X/S}(\eta)
$$
factors via the `generization map' $\Tropic_{X/S}^0(s) \to \Tropic_{X/S}^0(\eta_W)$ of \ref{eqn:generization_map}. In particular, $\lambda$ maps to $0$ in $\Tropic^0_{X/S}(\eta)$, which concludes the proof.
 \end{proof}

\begin{definition}\label{def:MW_tropic}
Let $X/S$ be a log curve. We define a sheaf of sets $\ca F_{X/S}$ on $\LSchet{S}$ as follows. If $S$ is locally of finite type and $T/S$ is of finite type, then $\c F_{X/S}(T)$ is a system of elements $a\in \Tropic^0(t)$, one for each geometric point $t\to T$, and compatible with the generization maps \ref{eqn:generization_map}. For general $T$, $\c F(T)$ is the sheafification of the presheaf taking $T$ to the colimit of the $\c F_{X/S}(T_0)$, taken over all $T \to T_0 \to S$ with $T_0$ of finite type. For general $S$, by \ref{remark:log_curves_are_finitely_presented} we can pick a cartesian square
\begin{center}
	\begin{tikzcd}
		X \ar[r]\ar[d] & S \ar[d]\\
		X_0 \ar[r] & S_0
	\end{tikzcd}
\end{center}
with $S_0$ locally of finite type and define $\c F_{X/S}$ as the pullback of $\c F_{X_0/S_0}$. This is independent from the choice of cartesian square, and the formation of $\c F_{X/S}$ commutes with base change. When there is no ambiguity, we will write $\c F$ instead of $\c F_{X/S}$.

\end{definition}
This is the way in which the tropical Jacobian is defined in \cite{Molcho2018The-logarithmic}.

\begin{lemma}\label{lemma:tropjac_atomic}
The sheaves $ \Tropic^0_{X/S}$ and $\ca F$ are canonically isomorphic.
\end{lemma}

\begin{proof}
If $S$ is locally of finite type and $T$ is a $S$-log scheme of finite type, there is a natural map $\Phi(T)\colon\Tropic^0(T)\to \c{F}(T)$ taking $\alpha\in \Tropic^0(T)$ to the system $\{i^*\alpha \in \Tropic^0(t)\}_{i\colon t \to T}$ where $i\colon t\to T$ are all geometric points of $T$. Now, let $T$ be an arbitrary $S$-log scheme. Since $\ca H_{1,X/S}$ is locally free and finite over $\ul S$, we find that locally on $T$, $\Tropic^0_{X/S}(T)$ is the colimit of the $\Tropic^0_{X/S}(T_0)$, taken over all $T \to T_0 \to S$ with $T_0$ of finite type. Combining this with the fact that the formation of $\Tropic^0_{X/S}$ commutes with strict base change, we get a morphism of abelian sheaves $\Phi\colon\Tropic^0_{X/S}\to \c{F}$ for any log curve $X/S$.

We will show that $\Phi$ is an isomorphism. We may do so assuming that $S$ is of finite type. Then, working locally, it suffices to show that $\Phi(T)$ is an isomorphism when $T$ is nuclear.

Let $x$ be a geometric point of $T$ landing in the closed stratum $Z$. Consider the natural map $\psi\colon \c{F}(T) \to \Tropic^0(x)$. The composition $\psi\circ \Phi(T)$ is the restriction map $\Tropic^0(T)\to \Tropic^0(x)$, which is an isomorphism by \ref{lemma:atomic_bm}. It suffices therefore to show that $\psi$ is injective (in order to conclude that it is an isomorphism, and that therefore $\Phi(T)$ is an isomorphism as well).

The argument is essentially the same as the end of the proof of \ref{lemma:atomic_bm}. Let $y \to T$ be any geometric point, $W_0$ the stratum containing $y$, and $W$ the connected component of $W_0$ containing $y$. By condition 1) in the definition of nuclearity there exists an \'etale specialization $\eta\rightsquigarrow \zeta$ with $\eta$ in $W$ and $\xi$ in $Z$. By \ref{corollary:generization_tropic}, we obtain a canonical map 
$$\Tropic^0(x)=\Tropic^0(\zeta)\to \Tropic^0(\eta)=\Tropic^0(y).$$
The compatibility condition forces all elements of the system $\c{F}(T)$ to be determined by the element belonging to $\Tropic^0(x)$. This proves the injectivity.
\end{proof}

\subsection{The logarithmic Jacobian}
Let $S=(\Spec k, M)$ be a logarithmic geometric point, with chart by a monoid $\oM$, and $\pi\colon X\to S$ a log curve. We consider $M^{\gp}_X$-torsors on $X$ for the strict \'etale topology. Given such a torsor $L$, we denote by $\overline L$ the image of its isomorphism class via $H^1(X,M^{\gp}_X)\to H^1(X,\oM^{\gp}_X)$.

There is a natural surjective map
$$\Hom(H_1(\mathfrak X),\oM^{\gp})=H^1(X,\pi^{-1}\oM^{\gp})\to H^1(X,\oM^{\gp}_X).$$
The first equality is because $\oM^{\gp}$ is torsion-free. Surjectivity is due to the fact that $\oM^{\gp}_{X/S}$ is supported on $X^{nsm}$ and therefore $H^1(X,\oM^{\gp}_{X/S})=0$.

We say that $L$ has \textit{bounded monodromy} if some (equivalently, any) preimage of $\overline L$ in $\Hom(H_1(\mathfrak X),\oM^{\gp})$ has bounded monodromy as in \ref{def:bm}.

Let now $\pi\colon X\to S$ be a log curve over a general log base; we say that a $M^{\gp}_X$-torsor $L$ has bounded monodromy, if for every strict geometric point $s\to S$, the restriction $L_s$ has bounded monodromy. Such a torsor is called a \textit{logarithmic line bundle}.

\begin{definition}
The logarithmic Picard stack is the stack $\mathbf{LogPic}_{X/S}$ on $(\cat{LSch}/S)_{\et}$ with sections 
\begin{align*}
\mathbf{LogPic}_{X/S}(T) = \{\mbox{log line bundles on }X_T=X\times_ST\}
\end{align*} 
The logarithmic Picard sheaf $\Logpic_{X/S}\subset R^1\pi_*\bb G_{m,X}^{log}$ is the sheafification of the functor of isomorphism classes of $\mathbf{LogPic}_{X/S}$.
\end{definition}

\begin{remark}
To any line bundle $L$ on a log curve $X\to S$ with associated $\ca O_X^\times$-torsor $L^\times$, one can associate the torsor $L^\times \otimes_{\mathcal O^{\times}_X}M^{\gp}_X$. Since its image in $H^1(X,\oM^{\gp}_X)$ vanishes, this has bounded monodromy and is therefore a log line bundle. There is therefore a natural map $\frak f_*\mathbf{Pic}_{X/S} \to \mathbf{LogPic}_{X/S}$.
\end{remark}

Consider the diagram
\begin{center}
\begin{tikzcd}
\frak f_*\mathbf{Pic}_{X/S} \ar[r]\ar[d] & \mathbf{LogPic}_{X/S} \ar[d, dashed, "\deg"] \\
\ZZ[\Irr_{X/S}] \ar[r, "\Sigma"] & \ZZ
\end{tikzcd}
\end{center}
where $\mathbf{Pic}_{X/S}$ is the Picard stack, the left vertical map is the multidegree map, $\Sigma$ is the sum, and the top horizontal map associates to a $\mathcal O_X^{\times}$-torsor $L$ the log line bundle $L\otimes_{\mathcal O^{\times}_X} M^{\gp}_X$. It is shown in \cite[4.5]{Molcho2018The-logarithmic} that there is a unique degree map making the diagram commute. 
\begin{definition}
We define $\mathbf{LogPic}^0_{X/S}$ to be the substack of $\mathbf{LogPic}_{X/S}$ of logarithmic line bundles of degree zero, and similarly for the sheaf $\Logpic^0_{X/S}\subset \Logpic_{X/S}$, which  is called \textit{logarithmic Jacobian}.
\end{definition}

For the convenience of the reader we recall from \cite{Molcho2018The-logarithmic} a list of properties of the logarithmic Picard and Jacobian that we will use.

\begin{property}\cite[Section 4.14]{Molcho2018The-logarithmic}.\label{property:tropicalization_map}
There is a natural morphism 
 $\Logpic^0_{X/S} \rightarrow \Tropic^0_{X/S}$, the ``tropicalization", which fits into a short exact sequence  
\begin{align}\label{exseq}
0 \rightarrow \frak f_*\Pic^0_{\ul X/\ul S}\to \Logpic^0_{X/S} \rightarrow \Tropic^0_{X/S} \to 0
\end{align}
\noindent where $\Pic^0_{\ul X/\ul S}$ is the generalized Jacobian, i.e. the sheaf of line bundles of degree zero on every irreducible component, which is representable by a semiabelian scheme on $\ul S$. Moreover, $\frak f_*\Pic^0_{\ul X/\ul S}$ is representable by a semiabelian scheme with pullback log structure from $S$.  
\end{property}

\begin{property}\cite[Proposition 4.3.2]{Molcho2018The-logarithmic}.
\label{property: line_bundle_representative}The bounded monodromy condition has the following concrete interpretation: an $M_{X}^\gp$-torsor has bounded monodromy if and only if, \'etale-locally on $S$, there exist log modifications $S'\to S$ and $Y\to X\times_S S'$ such that $Y\to S'$ is a log curve and the induced $M_{Y}^\gp$-torsor can be represented by a line bundle on $Y$. 
\end{property}

\begin{property}\cite[Corollary 4.4.14.1]{Molcho2018The-logarithmic}.\label{property:logpic_modifications}
If $Y \to X$ is a log modification such that $Y/S$ is a log curve, the induced maps $\Logpic^0_{X/S}\to \Logpic^0_{Y/S}$ and $\Tropic^0_{X/S}\to \Tropic^0_{Y/S}$ are isomorphisms. Although the statement for $\Tropic^0$ is not explicitly stated in \cite{Molcho2018The-logarithmic}, this follows from Corollary 4.4.14.1 \textit{loc. cit.} together with the fact that $\Pic^0_{X/S}\to \Pic^0_{Y/S}$ is an isomorphism and the exact sequence \ref{exseq}.
\end{property}

\begin{property}\cite[Section 4.4]{Molcho2018The-logarithmic}.\label{property:log_pic_is_a_log_stack}
The stack $\mathbf{LogPic}_{X/S}$ is a stack in the (strict) \'etale topology. However, if the base $S$ is logarithmically regular, $\mathbf{LogPic}_{X/S}$ is also a stack on the small \emph{log} \'etale site of $S$, and in fact on the site generated by the small log \'etale site and arbitrary root stacks (perhaps of order not prime to the characteristic). 
\end{property}

\begin{property}\cite[Theorem 4.10.1]{Molcho2018The-logarithmic}.\label{property:LogPic_proper}
The stack $\mathbf{LogPic}_{X/S}$ satisfies the valuative criterion for properness for log schemes: it has the unique lifting property with respect to valuation rings $R$ whose log structure is the direct image of a valuative log structure on the fraction field.
\end{property}

\begin{property}\cite[Theorem 4.13.1]{Molcho2018The-logarithmic}.\label{property:logpic_smooth}
The stack $\mathbf{LogPic}_{X/S}$ is logarithmically smooth over $S$, meaning that it is locally of finite presentation and satisfies the infinitesimal lifting criterion for strict square $0$ extensions, as in the original definition of \cite{Kato1989Logarithmic-str}.
\end{property}

In general, $\mathbf{LogPic}_{X/S}$ (resp. $\Logpic_{X/S}$) is not representable by a log algebraic stack (resp. log algebraic space), as the following example shows:

\begin{example}[\cite{IIKajiwara2008Logarithmic-a}, section I]\label{ex:tate_curve}
Let $K$ be a complete discrete valuation field, with ring of integers $\mathcal O_K$, $\pi\in \mathfrak m_K$ a uniformizer, $q=\pi^2$ and $E_q$ the corresponding Tate elliptic curve. Any proper log smooth model $\mathcal E_q$ of $E_q$ over $\mathcal O_K$ gives the same logarithmic Jacobian $\Logpic^0_{\mathcal E_q}$, and comes by \ref{thm:intro_logNMP} with a unique (birational) morphism $\mathcal E_q\to \Logpic^0_{\mathcal E_q}$. The minimal regular model $\mathcal E^{min}_q$ has two irreducible components over the residue field; blowing down any of them to a point gives two log smooth models $\mathcal E^1_q,\mathcal E^2_q$. As both $\mathcal E^1_q$ and $\mathcal E^2_q$ map to $\Logpic^0_{X/S}$, we see that if the latter were a log algebraic space, its closed fiber would have to be a point, contradicting the log smoothness.
\end{example}

\section{Strict log and tropical Jacobian and representability}\label{sec:representability}

\begin{definition}
Let $X\to S$ be a log curve. Recall the functor $\frak s_*$ from \ref{subs:functors}. We define the \textit{strict logarithmic Jacobian} $\sLPic^0_{X/S}$ to be $\frak s_*\Logpic^0_{X/S}$, a sheaf on the big \'etale site $\Schet{\ul S}$. Similarly, we define the \textit{strict tropical Jacobian} $\sTPic^0_{X/S}$ to be $\frak s_*\Tropic^0_{X/S}$
\end{definition}
Although these are sheaves on $\Schet{\ul S}$, they do not depend only on the prestable curve $\ul X/\ul S$ but also on the log structure $M_S$ (see \ref{example:trojac_depends_on_log_structure}). 

\begin{remark}
The article \cite{Molcho2018The-logarithmic} is by no means the first time the term ``logarithmic Jacobian" appears in the literature. To our knowledge the earliest appearances of the term are in \cite{Kato1989Logarithmic-deg,Kajiwara1993Logarithmic-c,Illusie1994Logarithmic-spa}; a notion of logarithmic Jacobian was then studied, among others, by Olsson in \cite{Olsson2004} and Bellardini in \cite{Bellardini2015On-the-Log-Pica}. Their notion is from the beginning a notion for schemes, and thus closer to our strict log Jacobian than to log Jacobian of \cite{Molcho2018The-logarithmic}. However, even so there are differences -- for instance, in Olsson's and Bellardini's work the log structure on $S$ is restricted, and the subset of $M_X^\gp$ torsors they consider is different from the subgroup of all bounded monodromy torsors. In particular, the objects constructed by Bellardini and Olsson will in general not satisfy the N\'eron mapping property over higher-dimensional bases. 
\end{remark}

There is an obvious modular interpretation for $\sLPic^0_{X/S}$, as the sheafification of the functor associating to a map of schemes $a\colon T\to \ul S$ the set $$\{\mbox{log line bundles } L \mbox{ on } X\times_S \frak sT \mbox{ of degree zero }\}/\cong$$

By exactness of the functor $\frak s_*$, the exact sequence \ref{exseq} yields an exact sequence
\begin{equation}\label{exact_sequence_strict}
0 \to \Pic^{0}_{X/S}\to \sLPic^0_{X/S}\to \sTPic^0_{X/S}\to 0
\end{equation}

\begin{lemma}\label{lemma:sTPic_basechange}
Let $X/S$ be a log curve and $f\colon T\to \ul S$ a morphism of schemes. Then 
$$\sTPic^0_{X/S}\times_{\ul S} T=\sTPic^0_{X_{\frak sT}/\frak sT} \;\;\; \text{and} \;\;\; \sLPic^0_{X/S}\times_{\ul S} T=\sLPic^0_{X_{\frak sT}/\frak sT}. $$ 
\end{lemma}
\begin{proof}
We have $\frak sT=\frak f_*T$, so $\frak s_*\frak sT=T$.  The formation of the tropical Jacobian commutes with base change and therefore \begin{equation}\label{eq:tropic_basechange}\Tropic^0_{X/S}\times_S\frak sT=\Tropic^0_{X_{\frak sT}/\frak sT}.
\end{equation} Now the right adjoint functor $\frak s_*$ commutes with products, and applying it to \ref{eq:tropic_basechange} yields the result. The same proof works for $\sLPic^0_{X/S}$.
\end{proof}

Recall that the stack $\textbf{LogPic}^0_{X/S}$ is proper and log smooth over $S$ (\ref{property:LogPic_proper} and \ref{property:logpic_smooth}), but in general not algebraic. The main purpose of this section is to prove the following:
\begin{theorem}\label{thm:sPic_representable}
Let $X/S$ be a log curve. Then 
\begin{enumerate}
\item
$\on{sTPic}^0_{X/S}$ is representable by a quasi-separated \'etale algebraic space over $\ul S$.
\item
$\on{sLPic}^0_{X/S}$ is representable by a quasi-separated smooth algebraic space over $\ul S$. 
\end{enumerate}
\end{theorem}

\begin{proof}

Part (ii) of the theorem is immediate from part (i) and the exact sequence  \ref{exact_sequence_strict}, as it realises $\on{sLPic}^0$ as a $\Pic^0$-torsor over $\on{sTPic}^0$, which makes it representable by a smooth separated algebraic space over $\on{sTPic}^0$. Let us prove part (i). Representability by a quasi-separated algebraic space is \'etale local on the target, and $X/S$ is of finite presentation, so we reduce via \ref{lemma:sTPic_basechange} to the case where $S$ is a log scheme of finite presentation (in particular Noetherian).

We reduce by virtue of \ref{cor:loc_constructible_implies_representable_big} to checking that the sheaf $\c{F} = \Tropic^0_{X/S}$ is locally constructible, i.e. that the canonical morphism 
$\alpha\colon \frak i^*\frak i_*\c{F}\to \c{F}$
is an isomorphism. 

It suffices to show that for any morphism $ T\to \ul S$ of schemes, the restriction of $\alpha$ to the small \'etale site over $T$ is an isomorphism. That is, that for any $T\to \ul S$ and any geometric point $t$ of $ T$, the map
$$\colim_{t\to V \to T} \frak i^*\frak i_*\c{F}(V) \to \colim_{t\to V\to T}\c{F}(V)$$
is an isomorphism, where the colimits are over factorizations $t \to V \to T$ with $V\to T$ \'etale. The right hand side is $\c{F}(t)$, by \ref{lemma:atomic_bm} and the fact that factorizations such that $V$ (with pullback log structure from $S$) is a nuclear neighbourhood of $t$ form a cofinal system. The left hand side becomes 
$$\colim_{t\to V\to T}\colim_{V\to W\to S}\c{F}(W)$$
where $V\to T$ and $W\to S$ are \'etale. This can in turn be replaced by the colimit of $\ca F(W)$ over the diagrams of the form 
\begin{center}
\begin{tikzcd}
t \ar[r] &V \ar[r] \ar[d] & T \ar[d] \\
&W \ar[r] & S
\end{tikzcd}
\end{center}
with $V\to T$ and $W\to S$ \'etale, where $t \to T$ and $T\to S$ are fixed and the remainder is allowed to vary. But this is simply the colimit over the factorizations $t\to W \to S$ with $W\to S$ \'etale. Since those factorizations with $W$ a nuclear neighbourhood of $t$ in $S$ form a cofinal system, by \ref{lemma:atomic_bm} the colimit is equal to $\ca F(t)$.
\end{proof}
 \subsection{Examples of strict logarithmic Jacobians}
 \begin{example}
 Let $S$ be the spectrum of a discrete valuation ring with divisorial log structure, and $X/S$ a log curve. Call $\eta$ the generic point of $S$. The closed fiber of the strict tropical Jacobian is identified with the finite \'etale group scheme of components of the N\'eron model of $\Pic^0_{X_\eta/\eta}$, and $\sLPic^0_{X/S}$ is the N\'eron model itself. This will be shown in greater generality in \ref{coro:NMP_strict_log}. Another explicit description of the N\'eron model is known in this case: after a finite sequence of log blowups of $X$, we obtain a new log curve $X'/S$ whose nodes all have length $1\in\NN=\o M_S(S)$. By \cite[9.5, Theorem 4]{Bosch1990Neron-models}, the N\'eron model of $\Pic^0_{X_\eta/\eta}$ is the quotient of the Picard group of degree $0$ line bundles $\Pic^{tot 0}_{X'/S}$ by the closure of its unit section. The relation between $\sLPic^0_{X/S}$ and quotients of Picard spaces is explored further in \ref{sec:etale_closure}.
 \end{example}
\begin{example}\label{example:1-gon}
Let $S=\Spec k[[u,v]]$, $D \subset S$ be defined by $uv=0$, and $j\colon U=S\setminus D\into S$. Let $E/S$ be the degenerate elliptic curve in $\bb P^2_S$ with equation
$$y^2z=x^3+x^2z+uvz^3$$
We make $E\to S$ into a log smooth morphism by putting the log structures associated to the divisors $D$ and $E\times_SD$ onto $S$ and $E$ respectively. Let $s$ be a geometric point of $S$; there are essentially three possible structures for the tropical curve $\mathfrak X_s$ and for $\Tropic^0(\mathfrak X_s)$:
\begin{itemize}
\item if $s$ lands in $U$, then $\mathfrak X_s$ consists of a single vertex and $\Tropic^0(\mathfrak X_s)=0$;
\item if $s$ lands in $D\setminus \{u=0,v=0\}$ then $\oM_{S,s}\cong \mathbb N$ and $\mathfrak X_s$ consists of a vertex and a loop labelled by $1\in \mathbb N$; then $\Tropic^0(\mathfrak X_s)=0$;
\item if $s$ maps to $(0,0)$ then $\oM_{S,s}\cong \mathbb N^2$ and $\mathfrak X_s$ consists of a vertex and a loop labelled by $(1,1)\in \mathbb N^2$; then $\Tropic^0(\mathfrak X_s)\cong \mathbb Z$.
\end{itemize}
It follows that $\sTPic^0_{X/S}$ is an \'etale group space with fiber $\bb Z$ at the closed point of $S$ and $0$ everywhere else. In particular, $\sLPic^0_{X/S}$ is not quasi-compact. It will follow from \ref{prop:separated_model} that it is not separated either.
\end{example}

\begin{example}\label{example:trojac_depends_on_log_structure}
The tropical Jacobian of a log curve depends on the log structures on $C/S$, and not only on the underlying scheme map $\ul C/\ul S$. Keeping the notations of \ref{example:1-gon}, $E/S$ comes via base-change from a log curve $E_0/S_0$, where $S_0$ is $\ul S$ with log structure given by
\[
(k[[u,v]])^\times \oplus \NN uv \to k[[u,v]].
\]
The tropical Jacobian of $E/S$ at the closed point is a free abelian group of rank $1$, while that of $E_0/S_0$ is trivial. In this example, $E_0 \to S_0$ coincides with the log curve $E^\# \to S^\#$ provided by \ref{proposition:special_log_structure}.
\end{example}

\section{The saturation of $\Pic^0$ in $\sLPic^0$}\label{sec:saturation}
We have seen in \ref{example:1-gon} that the \'etale algebraic space $\sTPic^0_{X/S}$ does not in general have finite fibers, and that consequentially $\sLPic^0_{X/S}$ is not in general quasi-compact. In this section we introduce a new quasi-compact quasi-separated (\emph{qcqs}) smooth algebraic space naturally associated to the log curve $X/S$ and sitting in between $\Pic^0_{X/S}$ and $\sLPic^0_{X/S}$.

 Consider the subsheaf on $\LSchet{S}$ of the tropical Jacobian
$$\Tropic^{tor}_{X/S} \subset \Tropic^0_{X/S}$$ of torsion elements, and its strict version $\sTPic^{tor}_{X/S}:=\frak s_*\Tropic^{tor}_{X/S}$ of torsion elements, 
\begin{definition}
We define the \textit{saturated Jacobian} $\Pic^{sat}_{X/S}$ to be the preimage of $\Tropic^{tor}_{X/S}$ via the map $\Logpic^0_{X/S}\to \Tropic^0_{X/S}$. Similarly, we define the \textit{strict saturated Jacobian} $\sPic^{sat}_{X/S}$ to be the preimage of $\sTPic^{tor}_{X/S}$ via the map $\sLPic^0_{X/S}\to \sTPic^0_{X/S}$.
\end{definition}

The exact sequence \ref{exseq} restricts to an exact sequence
$$0\to \frak f_*\Pic^0_{X/S}\to  \Pic^{sat}_{X/S}\to \Tropic^{tor}_{X/S} \to 0.$$

We apply the exact functor $\frak s_*$ to find an exact sequence in $\shSchet{S}$

$$0\to \Pic^0_{X/S}\to  \sPic^{sat}_{X/S}\to \sTPic^{tor}_{X/S} \to 0.$$

\begin{remark}
The strict saturated Jacobian has the following modular interpretation: it is the sheafification of the presheaf of abelian groups on $\Schet{\ul S}$ whose $T$-sections are the isomorphism classes of log line bundles $L$ of degree zero on $X\times_S sT$ such that some positive power of $L$ is a line bundle. This line of thought is pursued further in \cite{Holmes2022Logarithmic-mod}. 
\end{remark}

\begin{lemma}\label{lemma:sTPic_qf}
The inclusion $\sTPic^{tor}_{X/S}\to \sTPic^0_{X/S}$ is an open immersion; in particular $\sTPic^{tor}_{X/S}$ is representable by a quasi-separated $\ul S$-\'etale group algebraic space. Moreover $\sTPic^{tor}_{X/S}$ is quasi-finite over $\ul S$.
\end{lemma}
\begin{proof}
Denoting by $\sTPic[n]_{X/S}$ the subsheaf of $n$-torsion elements, for the first part of the statement it suffices to show that for every $n\in \ZZ$ the map $j\colon \sTPic[n]_{X/S}\to \sTPic^0_{X/S}$ is an open immersion. We have a pullback square
\begin{equation}
 \begin{tikzcd}
  \sTPic[n]_{X/S} \arrow[r]\arrow[d, "j"] & S \arrow[d, "0"] \\
  \sTPic^0_{X/S} \arrow[r, "n"] & \sTPic^0_{X/S} 
\end{tikzcd}
\end{equation}

Since $\sTPic^0_{X/S}/\ul S$ is \'etale, the zero section $S \to \sTPic^0_{X/S}$ is an open immersion, hence so is $j$. 

Now $\sTPic^{tor}_{X/S}$ is open in $\sTPic^0_{X/S}/\ul S$, and the latter is quasi-separated by \ref{thm:sPic_representable}, hence $\sTPic^{tor}_{X/S}$ is quasi-separated. 

To prove that $\sTPic^{tor}_{X/S} \to \ul S$ is quasi-finite, we will show that it is quasi-compact with finite fibers. For $s\in \ul S$ a geometric point, the restriction of $\sTPic^{0}_{X/S}$ to $s$ has, by \ref{lemma:sTPic_basechange}, as group of $s$-points the quotient $\Hom(H_1(\mathfrak X_s),\oM^{\gp}_{s})^{\dagger}/H_1(\mathfrak X_s)$. This is finitely generated since $\oM^{\gp}_s$ is finitely generated, hence its torsion part is finite.

For quasi-compactness, the question being local on the base, we assume that $\ul S$ is affine and that $\o M_S$ has a global chart from an fs monoid. Then the stratification of $\ul S$ induced by $\oM^{\gp}_S$ has only finitely many strata $Z_i$, and these strata are affine. Therefore, it suffices to show that for each $i$, the preimage of $Z_i$ in $\sTPic^{tor}_{X_{Z_i}/Z_i}$ is quasi-compact. 

On each stratum $Z$, the sheaf $\Hom(\mathcal H_{1,X_Z/Z},\frak s_*\bb G^{trop}_{m,Z})^\dagger$ is locally constant by \ref{corollary:generization_tropic}, so $\sTPic^0_{X_Z/Z}$ is locally constant as well. Its torsion part $\sTPic^{tor}_{X_Z/Z}$ is then a locally constant sheaf of finite abelian groups and in particular a finite \'etale scheme over $Z$. It is therefore quasi-compact.
\end{proof}

\begin{corollary}
The inclusion $\sPic^{sat}_{X/S}\to \sLPic^{0}_{X/S}$ is an open immersion, and $\sPic^{sat}_{X/S}$ is representable by a  $\ul S$-smooth group algebraic space of finite presentation.
\end{corollary}
\begin{proof}
The first part of the statement follows from \ref{lemma:sTPic_qf} and base change. The representability then follows from \ref{thm:sPic_representable}.We deduce that $\sPic^{sat}_{X/S}$ is qcqs because it is a torsor over the quasi-finite and quasi-separated $\sTPic^{tor}_{X/S}$, under the quasi-compact separated group scheme $\Pic^0_{X/S}$.
\end{proof}
 
\begin{lemma}\label{lemma:model_universal_jacobian}
Let $S=\overline{\c{M}}_{g,n}$, the moduli stack of genus $g$ stable curves with $n$ marked points. Let $\c{X}_{g,n}$ be the universal curve. Endow both stacks with the divisorial log structure so as to make the universal curve into a log curve. Then $\sPic^{sat}_{\c{X}_{g,n}}=\Pic^0_{\c{X}_{g,n}}$.
\end{lemma}
\begin{proof}
For $s$ a geometric point of $\overline{\c{M}}_{g,n}$, the tropicalization $\mathfrak X$ of the fiber has each edge labelled by a distinct base element of the free monoid $\oM_{S,s}$. Thus $\Tropic^0(\mathfrak X)$ is torsion-free. 
\end{proof}

\section{The N\'eron mapping property of LogPic}\label{section:NMP}
In this section we prove the N\'eron mapping property for the logarithimic Jacobian, under the assumption that the base is log regular. 

\subsection{Classical and log N\'eron models}\label{def:NM}\label{def:NMP}
For comparison, we briefly recall the definitions of classical, non-logarithmic N\'eron models. 

\begin{definition}\label{definition_sch_neron_models}
Let $S$ be a scheme, $U\subset S$ a schematically dense open, and $\c{N}/S$ a category fibered in groupoids over $\cat{Sch}/S$. We say that $\c{N}/S$ has the \emph{N\'eron mapping property} with respect to $U\subset S$ if for every smooth morphism of schemes $T\to S$, the restriction map 
$\c{N}(T)\to \c{N}(T\times_{S} U)$
is an equivalence.

If in addition $\ca N$ is an algebraic stack and is smooth over $S$, we say $\ca N$ is a \emph{N\'eron model} of its restriction $\mathcal N_U/U$. 
\end{definition}

\begin{remark}
The definition of N\'eron models in \cite{Holmes2014Neron-models-an} and the classical one over Dedekind schemes require them to be separated. When the base is a Dedekind scheme and the generic fiber is a group scheme this is automatic by \cite[Theorem 7.1.1]{Bosch1990Neron-models}. However, over higher-dimensional bases this is not true, and non-separated N\'eron models exist in much greater generality than separated ones (the author of \cite{Holmes2014Neron-models-an} did not realise this at the time). From this perspective, \textit{loc. cit.} should be seen as investigating the existence of \emph{separated} N\'eron models. 
\end{remark}

This definition extends naturally into the logarithmic setting: 
\begin{definition}\label{def:log_NMP}
Let $S$ be a log scheme, $U\subset S$ strict schematically dense open, and $\c{N}/S$ a category fibered in groupoids over $\cat{LSch}/S$. We say that $\c{N}/S$ \emph{has the N\'eron mapping property} with respect to $U\subset S$ if for every log smooth morphism $T\to S$, the map
$\c{N}(T)\to \c{N}(T\times_SU)$
is an equivalence. 
\end{definition}

\begin{remark}
In order to qualify as a N\'eron model, the functor $\ca N$ should not only satisfy the N\'eron mapping property, but also be a sheaf for a suitable topology, and be representable in some suitable sense. Requiring representability as an algebraic space or stack with log structure is too restrictive; the log Jacobian does not satisfy these criteria, and in general the log N\'eron model will not exist as an algebraic space or stack with log structure. The most appropriate notions of representability are perhaps:
\begin{enumerate}
\item
for the functor, that it is a sheaf for the strict log \'etale topology, has diagonal representable by log schemes, and admits a log \'etale cover by a log scheme;
\item for the stack, that it is a sheaf for the strict log smooth topology, it has diagonal representable in the sense of (i), and admits a log smooth cover by a log scheme.
\end{enumerate}
The Log Picard space and stack do satisfy these additional conditions. This is shown in \cite{Molcho2018The-logarithmic}, except for verifying that LogPic is a stack for the log smooth topology (\cite{Molcho2018The-logarithmic} only verify it for the log etale topology), but this will be proven in the forthcoming \cite{Molcho2020The-Logarithmic}). 

One consequence of imposing the above assumptions is that N\'eron models are unique\footnote{Either up to unique isomorphism, for the functor, or up to 1-isomorphism which is itself unique up to a unique $2$-isomorphism, in the stack case. }, by analogous descent arguments to those for the classical case. For example, suppose that $\ca N$ and $\ca N'$ are log N\'eron models, and let $V \to \ca N$ be a log \'etale cover by a log scheme; then the map $V_U \to \ca N_U \iso \ca N'_U$ extends uniquely to a map $V \to \ca N'$ by the universal property of the latter, and a further application of uniqueness and the sheaf property shows that this descends to a map $\ca N \to \ca N'$. 
\end{remark}

\begin{remark}
In classical algebraic geometry the distinction between the smooth and \'etale topologies is often not so important since every smooth cover can be refined to an \'etale cover (see \cite[\href{https://stacks.math.columbia.edu/tag/055V}{Tag 055V}]{stacks-project}). In the logarithmic setting this is no longer true; for example the map 
\begin{equation}
\bb G_m \times \bb A^1 \to  \bb A^1 ; (x,t) \mapsto xt^p
\end{equation}
where $\bb G_m, \bb A^1$ have their toric log structure is a log smooth cover, but does not admit a refinement by a  log \'etale cover in characteristic $p$. It does admit a refinement by a Kummer cover, and this is in fact a general phenomenon; one can show using the results of \cite{adiprasito2018semistable} that every log smooth cover can be refined by a cover that is a composite of log \'etale covers and Kummer maps. 
\end{remark}


\begin{remark}\label{remark:NMP_strict}
Suppose $\mathcal N/S$ is a functor on $(\cat{LSch}/S)$ with the log N\'eron mapping property with respect to $U\subset S$. Then $\frak s_*\mathcal N$ has the N\'eron mapping property with respect to $\ul U\subset \ul S$ since, for every smooth morphism of schemes $g\colon T\to \ul S$, the morphism $(T,g^*M_S)\to S$ is log smooth. 
\end{remark}

\subsection{Log regularity}\label{sec:log_regular}
We will show that the logarithmic Jacobian of $X/S$ is a log N\'eron model when $S$ is log regular, a notion which we now recall. 

\begin{definition}[\cite{Kato1994Toric-singulari,Nizio2006Toric-singulari}]\label{subsection:log_regular} 
Let $S$ be a locally Noetherian fs log scheme. For a geometric  point $s$ of $S$, we denote the ideal generated by the image of $M_{S,{s}}\setminus \ca O_{S,s}^* \rightarrow \mathcal{O}_{S,{s}}$ by $I_{S,{s}}$. We say that $S$ is \emph{log regular} at $s$ if $\mathcal{O}_{S,{s}}/I_{S,{s}}$ is regular and $\dim \mathcal{O}_{S,{s}} = \dim \mathcal{O}_{S,{s}}/I_{S,{s}} + \textup{rk} \overline{{M}}_{S,{s}}^\gp$. We say that $S$ is \textit{log regular} if it is log regular at all geometric points.
\end{definition}

\begin{remark}
Any toric or toroidal variety with its natural log structure is log regular. In particular, a regular scheme equipped with the log structure from a normal crossings divisor is log regular. If $S$ is log regular then the locus on which the log structure is trivial is schematically\footnote{Log regular schemes are reduced, so density is equivalent to schematic density. } dense open. 
\end{remark}

\begin{lemma}[\cite{Nizio2006Toric-singulari}, Lemma 5.2]\label{lem:log_regular_characterisation} 
Let $S$ be a log regular scheme. Then the underlying scheme of $S$ is regular if and only if the characteristic sheaf $\oM_S$ is locally free. In this case, the log structure on $S$ is the log structure associated to a normal crossings divisor on $S$.  
\end{lemma}

\begin{lemma}[{\cite[theorem 11.6]{Kato1994Toric-singulari}}]\label{lemma:niziol}
Let $S$ be a log regular scheme, and $U\hra S$ the dense open locus where the log structure is trivial. Then $M_S = j_*\c{O}_U^* \times_{j_*\c{O}_U} \c{O}_S$ and $M_S^{\gp} = j_*\c{O}_U^*$, where $j: U \rightarrow X$ is the inclusion.
\end{lemma}

\subsection{The log Jacobian is a log N\'eron model}

\begin{theorem}\label{thm:Neron}
Let $X/S$ be a log curve over a log regular scheme, and $U\hra S$ the dense open locus where the log structure is trivial. The stacks $\mathbf{LogPic}_{X/S}$ and $\mathbf{LogPic}^0_{X/S}$ on $(\cat{LSch}/S)_{\et}$ are log smooth over $S$ and have the N\'eron mapping property with respect to $U$ (hence likewise for their sheafifications $\Logpic_{X/S}$, $\Logpic^0_{X/S}$).
\end{theorem}

\begin{proof}
The log smoothness is \ref{property:logpic_smooth}. We check the N\'eron mapping property. For any log smooth morphism $T \to S$, the restriction map $\Logpic_{X/S}(T)\to \Logpic_{X/S}(T\times_S U)$ is obtained, locally on $T$, from $\mathbf{LogPic}_{X/S}(T)\to \mathbf{LogPic}_{X/S}(T\times_S U)$ by taking isomorphism classes on both sides. Hence, if the stack $\mathbf{LogPic}_{X/S}$ has the mapping property, then so does the sheaf $\Logpic_{X/S}$. Likewise, $\Logpic^0_{X/S}$ has the N\'eron mapping property if $\mathbf{LogPic}^0_{X/S}$ does. Since the constant sheaf $\ZZ$ on $\cat{LSch}/S$ has the N\'eron mapping property, the same property for $\mathbf{LogPic}_{X/S}$ implies it for the kernel $\mathbf{LogPic^0}_{X/S}$ of the degree map. 

It remains to prove the property for $\mathbf{LogPic}_{X/S}$. Let $T \rightarrow S$ be a log smooth map, $V = T \times_S U$. Let $i: V \rightarrow T$, $j: X_V \rightarrow X_T$ denote the inclusions. We want to prove that the map
\begin{equation}\label{eq:restriction_logpic}
\mathbf{LogPic}_{X/S}(T)\to \mathbf{LogPic}_{X_U/U}(V)=\mathbf{Pic}_{X_U/U}(V)
\end{equation}
given by $P\mapsto j^*P$ is an equivalence.
We start with full faithfulness: let $P,Q$ be log line bundles on $X\times_ST$. It suffices to show that the map $\sheafisom(P,Q)\to j_*\sheafisom(j^*P,j^*Q)$ of isomorphism sheaves is itself an isomorphism. The natural map $M^{\gp}_{X_T}\to j_*\c{O}^{\times}_{X_V}$ is an isomorphism by log regularity of $X_T$ and \ref{lemma:niziol}. Hence the natural map $P \to j_*j^*P$ is a map of $M^{\gp}_{X_T}$-torsors, hence is an isomorphism (and similarly for $Q$). The map $\sheafaut(P)\to j_*\sheafaut(j^*P)$ is simply the natural map $M^{\gp}_{X_T}\to j_*\c{O}^{\times}_{X_V}$, hence is also an isomorphism. The map $$\sheafisom(P,Q)\to j_*\sheafisom(j^*P,j^*Q)$$ is then a map of $M^{\gp}_{X_T}$-pseudotorsors, so it suffices to show that $\sheafisom(P,Q)$ is non-empty whenever $j_*\sheafisom(j^*P,j^*Q)$ is. But given an isomorphism $j^*P \iso j^*Q$, taking $j_*$ gives an isomorphism 
\begin{equation*}
P = j_*j^*P \iso j_*j^*Q = Q
\end{equation*}
(the above argument can be summarised by saying that $j_*$ is a quasi-inverse to $j^*$).

We now prove essential surjectivity. Letting $L$ be a, $\ca O_{X_V}^*$-torsor on $X_V$, we will exhibit a $M_{X_T}^{\gp}$-torsor $P$ of bounded monodromy such that $j^*P=L$. The torsor $P$ will simply be $j_*L$; as a priori this is only a pseudo-torsor under $j_*\mathcal{O}_{X_V}^* = M_{X_T}^\gp$, and, if it is a torsor, it is not obvious that it has bounded monodromy, we will give a geometric construction of $P$, using the invariance of $\textbf{LogPic}$ under log blowups and taking root stacks. Let $D$ be a Cartier divisor on $X_V$ representing $L$. By \cite[Theorem 4.5]{adiprasito2018semistable}, there is a cover $u\colon T' \rightarrow T$, which is a composition of a log blowup and a root stack, such that $u$ restricts to an isomorphism over $V$, and a log modification $p\colon X'\to X_{T'}=X \times_T T'$ with $X' \rightarrow T'$ semistable (\cite[Definition 2.2]{adiprasito2018semistable}). In particular, $X' \rightarrow T'$ is a log curve, and the total space of $X'$ is regular. The schematic closure $\overline D$ of $D$ in $X'$ is therefore a Cartier divisor, and $\mathcal O(\overline D)$ is a line bundle whose associated $\ca O_{X'}^*$-torsor extends $L$. By \ref{property:logpic_modifications}, it follows that there is a $M_{X_{T'}}^{\gp}$-torsor $P'$ of bounded monodromy on $X_{T'}$ such that $j^*P'=L$. Then, by the full faithfulness of \ref{eq:restriction_logpic} and the fact $\textbf{LogPic}_{X/S}$ is a stack for the log \'etale topology (\ref{property:log_pic_is_a_log_stack}), this $P'$ descends to a $M_{X_T}^{\gp}$-torsor $P$ with bounded monodromy. 
\end{proof}

\begin{remark}
It will be shown in \cite{Molcho2020The-Logarithmic} that, when $S$ is logarithmically smooth, $\mathbf{LogPic}_{X/S}$ is a stack, not only for the log \'etale topology, but for the topology generated by all logarithmically flat morphisms, from which it follows that $\mathbf{LogPic}_{X/S}$ and $\mathbf{LogPic}^0_{X/S}$ are in fact the N\'eron models of their restrictions to $U$. 
\end{remark}

\begin{corollary}\label{coro:NMP_strict_log}
Let $X/S$ be a log curve over a log regular scheme. The smooth algebraic spaces $\sLPic_{X/S}$ and $\sLPic^0_{X/S}$ are N\'eron models of $\Pic_{X/S}$ and $\Pic^0_{X/S}$ respectively.
\end{corollary}
\begin{proof}
This follows from \ref{thm:Neron} together with \ref{remark:NMP_strict}.
\end{proof}

\begin{corollary}\label{corollary:torsion_NMP_sat_logpic^0}
The strict saturated Jacobian $\sPic^{sat}_{X/S}$ satisfies the N\'eron mapping property for torsion sections. 
That is, for every $T\to S$ smooth and $L\colon T_U\to \Pic^0_{X/S}$ of finite order, there exists a unique extension to a map $T\to \sPic^{sat}_{X/S}$. In particular, for every prime $l$
$$T_l\sPic^{sat}_{X/S}=j_*T_l\Pic^0_{X_U/U}$$
as sheaves on the \'etale site over $\ul S$, where $j\colon U\to S$ is the inclusion and $T_l$ indicates the $l$-adic Tate module.
\end{corollary}
\begin{proof}
By \ref{thm:Neron}, $L$ extends uniquely to $\ca L$ in $\sLPic^0_{X/S}(T)$. As the image of $\ca L$ in $\sTPic^0_{X/S}(T)$ is torsion, $\ca L$ actually lies in $\sPic^{sat}_{X/S}$
\end{proof}

\section{Models of $\on{Pic}^0$}\label{sec:models}
We saw in \ref{section:NMP} that the logarithmic Jacobian and its strict version are N\'eron models.   However, while $\Logpic^0_{X/S}$ is proper (in a suitable sense), its strict version $\sLPic^0_{X/S}$ is in general neither quasi-compact, separated, nor universally closed. 

In this section we undertake the study of smooth \textit{separated} group-models of $\Pic^0_{X_U}$. We establish a precise correspondence between such models and subgroups of the strict tropical Jacobian $\sTPic^0_{X/S}$.

\subsection{A tropical criterion for separatedness}

We start by considering a tropical curve $\mathfrak X=(V,H,r,i,\ell)$ metrized by a sharp monoid $\oM$. For a given a monoid homomorphism $\phi\colon \oM\to \oN$ we denote by $\mathfrak X_{\phi}$ the induced tropical curve metrized by $\oN$. We recall that by assumption all monoids we work with are fine and saturated.

\begin{lemma}\label{lemma:map_tropic}
Let $\mathfrak X$ be a tropical curve metrized by a sharp monoid $\oM$ and $\phi\colon \oM\to \oN$ a map of monoids not contracting any edge of $\mathfrak X$. The induced homomorphism of tropical Picard groups  
$\Tropic^0(\mathfrak X)\to \Tropic^0(\mathfrak X_{\phi})$ has free kernel.
\end{lemma}

\begin{proof}
The map $\phi$ induces an isomorphism $H_1(\mathfrak X)\to H_1(\mathfrak X_{\phi})$. By the snake lemma, the kernel of the map of tropical Jacobians is equal to the kernel of the map 
$$\Hom(H_1(\mathfrak X),\oM^{\gp})^{\dagger}\to \Hom(H_1(\mathfrak X),\oN^{\gp})^{\dagger}$$
induced by $\phi^{\gp}\colon\oM^{\gp}\to \oN^{\gp}$. Because $\oM$ is fine and saturated, $\oM^{\gp}$ is free, hence so is $\Hom(H_1(\mathfrak X),\oM^{\gp})$ and any subgroup of it.
\end{proof}

\begin{lemma}\label{coro:finite_tropjac_injective_maps}
Let $\mathfrak X$ be a tropical curve metrized by a sharp monoid $\oM$, let $\Psi$ be an abelian group, and let $\alpha\colon \Psi\to \Tropic^0(\mathfrak X)$ be a homomorphism. The following are equivalent:
\begin{enumerate}
\item $\alpha$ is injective and $\Psi$ is finite;
\item for every monoid $P$ and homomorphism $\phi\colon \oM\to P$ not contracting any edge of $\mathfrak X$, the composition $\Psi \to \Tropic^0(\mathfrak X)\to \Tropic^0(\mathfrak X_{\phi})$ is injective;
\item for every monoid homomorphism $\phi\colon \oM\to \NN$ not contracting any edge of $\mathfrak X$, the composition $\Psi \to \Tropic^0(\mathfrak X)\to \Tropic^0(\mathfrak X_{\phi})$ is injective.
\end{enumerate}
\end{lemma}
\begin{proof}
\leavevmode
\begin{itemize}
\item[] $({\rm i})\Rightarrow ({\rm ii})$: The kernel of $\Tropic^0(\mathfrak X)\to \Tropic^0(\mathfrak X_{\phi})$ is free by \ref{lemma:map_tropic}. Since $\Psi$ is finite, the intersection of the kernel with $\Psi$ is zero.

\item[] $({\rm ii})\Rightarrow ({\rm iii})$ is clear. 

\item[] $({\rm iii})\Rightarrow ({\rm i})$: We first establish the existence of a monoid homomorphism $\phi\colon \oM\to \NN$ not contracting any edge of $\mathfrak X$: since $\overline{M}$ is fine and saturated, it is the intersection of a cone $\sigma \subset P^{\gp}\otimes \bb{R}$ with $P^{\gp}$ for a finitely generated lattice $P^\gp$. Then, any integral element in the interior of the dual cone $\sigma^{\vee}\subset \Hom(P^{\gp},\bb R)$ will not contract any non-zero element of $M$ and thus gives rise to such a map. Injectivity of $\alpha$ thus follows. Now $\Tropic^0(\mathfrak X_{\phi})$ is finite since it is the cokernel of the homomorphism $H_1(\frak X_{\phi})\to \Hom(H_1(\frak X_{\phi}),\bb Z)$ induced by the monodromy pairing, which is injective by \cite[Corollary 3.4.8]{Molcho2018The-logarithmic}. Hence, $\Psi$ is finite as well.\qedhere
\end{itemize}
\end{proof}

We fix a log curve $X/S$ with $\ul S$ locally noetherian. We introduce the category $\cat{Et}$ whose objects are pairs $(\Psi,\alpha)$ of an \'etale group algebraic space $\Psi/S$ and a homomorphism $\alpha\colon \Psi\to \sTPic^0_{X/S}$. Similarly, we let $\cat{Sm}$ be the category of pairs $(\mathcal G,\alpha)$ of a smooth group algebraic space $\mathcal G/S$ and a homomorphism $\alpha\colon \ca G\to \sLPic^0_{X/S}$.

There is an obvious base change functor 
\begin{align}\label{eqn:functor}
F\colon \cat{Et} &\xrightarrow{\hspace*{3cm}} \cat{Sm} \\
\left(\Psi \to \sTPic^0_{X/S}\right) &\xmapsto{\hspace{2cm}} \left(\Psi\underset{\sTPic^0_{X/S}}\times\sLPic^0_{X/S} \to \sLPic^0_{X/S}\right) \nonumber
\end{align}

Recall the functor $\frak f_*$ introduced in \ref{subs:functors}. Notice that for an object $(\Psi,\alpha)$ of $\cat{Et}$, we have a natural map $\frak f_*\Psi\to \frak f_*\sTPic^{0}\to \Tropic^0$ and similarly for an object of $\cat{Sm}$. 

The next proposition is the key statement of the section.

\begin{proposition}\label{prop:separated_model}
Let $X/S$ be a log curve with $\ul S$ locally noetherian. Let $(\Psi,\alpha)\in\cat{Et}$, with $\c{G}\to \sLPic^0_{X/S}$ its image under the functor $F$. Consider the two conditions:

\begin{enumerate}
\item \label{cond:1} $\alpha\colon \Psi\to \sTPic^0_{X/S}$ is an open immersion and $\Psi/S$ is quasi-finite;
\item \label{cond:2} $\c{G}/S$ is separated. 
\end{enumerate}
Then \oref{cond:1} implies \oref{cond:2}. Moreover, if $(S,M_S)$ is log regular and $\Psi$ has trivial restriction to the dense open $U\subset S$  where $\oM_S$ vanishes, then \oref{cond:2} implies \oref{cond:1} as well, and the two conditions are equivalent to 
\begin{enumerate}[resume]
\item 
\label{cond:3} $\c{G}/S$ is separated and quasi-compact.
\end{enumerate}
\end{proposition}

\begin{remark}
The restriction of $\sLPic^0_{X/S}$ to $U$ is naturally identified with $\Pic^0_{X_U/U}$, and therefore $\sTPic^0_{X/S}$ restricts to $\{0\}$. The fact that $\Psi_{U}=\{0\}$ implies that $\mathcal G_{U}= \Pic^0_{X_U/U}$.  
\end{remark}
\begin{proof}
We prove that \oref{cond:1} implies \oref{cond:2}. Since $\ul S$ is locally noetherian, $\Psi/S$ is quasi-separated, hence so is its base change $\ca G/S$. By \cite[\href{https://stacks.math.columbia.edu/tag/0ARI}{Tag 0ARI}]{stacks-project} we may check the valuative criterion for a strictly henselian discrete valuation ring $V$; we write $\eta$ for its generic point and $s$ for its closed point. Fix a map $a\colon \eta\to \mathcal G$ and two lifts to $b,b'\colon V\to \mathcal G$; using the group structure of $\mathcal G$, we may assume $b'=0$ and therefore $a=0$.  

Denote by $d$ the map from $V$ to $S$; we endow $\eta$ with the log structure $M_{\eta}$ pulled back from $S$, and we write $M_V$ for the maximal extension of $M_{\eta}$ to a log structure on $V$. We obtain a commutative diagram

\begin{center}
\begin{tikzcd}
(\eta,M_\eta) \ar[r] \ar[d] & (\eta,M_\eta) \ar[r, "0"] \ar[d] & \frak f_*\mathcal G \ar[r] \ar[d] & \Logpic^0_{X/S} \ar[d] \\ (V,M_V) \ar[r] & (V,d^*M_S) \ar[r] \ar[ru,shift left, "b"]\arrow[ru,shift right, swap, "0"] & S\ar[r] &  S
\end{tikzcd}
\end{center}

Let $\mathfrak{X}_s$ (resp. $\mathfrak X^V$) denote the tropicalization of $X_s$ metrized by $\o{M}_{S,s}$ (resp. $\o{M}_{V,s}$). Notice that the map $\o{M}_{S,s} \rightarrow \o{M}_{V,s}$ of characteristic monoids does not contract any edge of $\mathfrak{X}_s$, since $d^*M_S\to M_V$ is a map of log structures and therefore sends non-units to non-units. As $V$ is strictly henselian, we have by \ref{lemma:atomic_bm} that $\Tropic^0(V,M_V)=\Tropic^0(\mathfrak X^V)$ and $\Tropic^0(V,d^*M_S)=\Tropic^0(\mathfrak X_s)$.
Consider the commutative diagram whose rows are exact sequences:
\begin{align*}
\xymatrix{0 \ar[r] & \Pic^0_{X/S}(V) \ar[r] \ar[d] & \frak f_*\mathcal G(V,d^*M_S) \ar[r] \ar[d] & \frak f_*\Psi(V,d^*M_S) \ar[d]\ar[r]  & 0 \\
0 \ar[r] & \Pic^0_{X/S}(V) \ar[r] \ar[d] & \Logpic^0_{X/S}(V,d^*M_S) \ar[r] \ar[d] & \Tropic^0(\mathfrak{X}_s) \ar[d] \ar[r] & 0 \\ 0 \ar[r] & \Pic^0_{X/S}(V) \ar[r] & \Logpic^0_{X/S}(V,M_V) \ar[r] & \Tropic^0(\mathfrak{X}^V) \ar[r] & 0}
\end{align*} 
The element $b$ lies in $\frak f_*\mathcal G(V,d^*M_S)$. The following three facts imply that $b=0$: 
\begin{itemize}
\item By properness of $\Logpic^0_{X/S}$, the image of $b$ in  $\Logpic^0_{X/S}(V,M_V)$ is the trivial log line bundle, so $b$ maps to $0$ in $\Tropic^0(\mathfrak X^V)$. 
\item
Combining the assumption that \oref{cond:1} holds with \ref{coro:finite_tropjac_injective_maps} yields that the composition $$\frak f_*\Psi(V,d^*M_S) = \Psi(V)\to \sTPic^0(V)=\Tropic^0(V,d^*M_S)\to \Tropic^0(V,M_V)$$ is injective. 
\item If $b\in \frak f_*\c{G}(V,d^*M_S)$ lies in the image of $\Pic^0_{X/S}(V)$ then $b=0$ since $\Pic^0_{X/S}$ is separated. 
\end{itemize}

%
%
%

We move on to the next part of the statement, so from now on we suppose that $S$ is log regular and that $\Psi_U=\{0\}$. This in particular implies that $\mathcal G_U= \Pic^0_{X_U/U}$. It's clear that \oref{cond:3} implies \oref{cond:2}; if we show that \oref{cond:2} implies \oref{cond:1}, then we immediately obtain $\oref{cond:2}\Rightarrow \oref{cond:3}$. Indeed, quasi-finiteness of $\Psi$ together with the fact that $\Pic^0_{X/S}$ is quasi-compact, implies that $\mathcal G/S$ is quasi-compact.

It remains to prove that \oref{cond:2} implies \oref{cond:1}. Write $K$ for the kernel of $\Psi\to \sTPic^0_{X/S}$. It is \'etale, and since $\Psi_U=0$, $K$ also vanishes over the open dense $U\subset S$. Moreover, $K$ is identified with the kernel of $\c{G}\to \sLPic^0_{X/S}$; as $\c{G}/S$ is separated, so is $K$, hence $K$ is trivial. This shows that $\Psi\to \sTPic^0_{X/S}$ is an open immersion.

Now let $t\to S$ be a geometric point, with image $s\in S$. Because $S$ is locally noetherian, by a special case of \cite[7.1.9]{Grothendieck1961EGAII} there exists a morphism $Z\to S$ from the spectrum of a discrete valuation ring such that the closed point is mapped to $s$ and the generic point to $U$. Now consider the composition $\omega\colon Z^{sh}\to Z\to S$ with the strict henselization induced by $t\to s$. $\c{G}$ is a smooth, separated model of $\Pic^0_{X_U}$ and therefore $\omega^*\c{G}$ is a $Z^{sh}$-smooth separated model of $\omega^*\Pic^0_{X_U}$. The restriction of the latter to the generic point of $Z^{sh}$ is an abelian variety, which therefore admits a N\'eron model of finite type $\c{N}/Z^{sh}$ by \cite[Corollary 1.3.2]{Bosch1990Neron-models}. We claim that the natural map $\Phi\colon \omega^*\c{G}\to \c{N}$ is an open immersion. In the case of schemes this would follow immediately from \cite[Prop. 7.4.3]{Bosch1990Neron-models}. The same proof shows for algebraic spaces that the natural map on identity components $\omega^*\c{G}^0\to \c{N}^0$ is an isomorphism. This implies that $\Phi$ is flat, hence is an open immersion by \ref{lem:open_imm_of_flat_sep_lfp}. 


 In particular the map of fibers $\c{G}_t\to \c{N}_t$ is an open immersion.  As $\c{N}_t$ is of finite type (and $S$ is locally noetherian), so is $\c{G}_t$. Then by descent $\Psi_t$ is of finite type as well; as it is moreover \'etale, it is finite over $k(t)$. In particular the map $\Psi_t\to \sTPic^0_t$ factors via $\sTPic^{tor}_t$ (which is quasi-finite by \ref{lemma:sTPic_qf}); it follows that the open immersion $\Psi\to \sTPic^0$ factors via the open immersion $\sTPic^{tor}\to \sTPic^0$. The resulting open immersion $\Psi\to \sTPic^{tor}$ is quasi-compact (indeed $\sTPic^{tor}$ is locally noetherian since  $S$ is). This proves that $\Psi/S$ is quasi-finite.
\end{proof}

\begin{lemma}\label{lem:open_imm_of_flat_sep_lfp}
Let $f\colon X \to Y$ be a flat, separated, locally finitely presented morphism of algebraic spaces. Let $U \subseteq Y$ be open such that $f$ is an open immersion over $U$ and such that $f^{-1}U$ is schematically dense in $X$. Then $f$ is an open immersion. 
\end{lemma}
\begin{proof}We follow \cite[Lemma 2.0]{Lutkebohmert1993On-compactifica}. Replacing $Y$ by the open image of $f$, we may assume $f$ is faithfully flat. Base-changing along $f$, the assumptions are preserved and we may assume $f$ has a section $e\colon Y \to X$. Since $f$ is separated, the section $e$ is a closed immersion. But the open immersion $j\colon f^{-1}U \to X$ factors via $e$, so schematic density of $j$ implies that $e$ is an isomorphism. 
\end{proof}

The image of $\sTPic^{tor}_{X/S}$ under the functor $F$ is $\sPic^{sat}_{X/S}$, yielding:
\begin{corollary}
Let $X/S$ be a log curve with $\ul S$ locally noetherian. Then $\sPic^{sat}_{X/S}$ is separated.
\end{corollary}

\subsection{Equivalence of categories}
From this point until the end of the section we will assume that $S$ is log regular and denote by $U\subset S$ the open dense where $\oM_S$ is trivial. Recall that the \'etale algebraic space $\sTPic^0_{X/S}$ has trivial restriction to $U$ while $\sLPic^0_{X/S}$ is a smooth group space whose restriction to $U$ is naturally identified with $\Pic^0_{X_U/U}$.

We consider the full subcategory of $\cat{Et}$ 
\begin{align*}\cat{Et^{\circ}}:=&\{(\Psi,\alpha) \text{ with } \Psi \to S\text{ an \'etale group-algebraic space} \\& \text{such that } \Psi_U=\{0\}, \; \alpha\colon \Psi\to \sTPic^0_{X/S}  \text{ a homomorphism}\}.
\end{align*} 
and the full subcategory of $\cat{Sm}$
\begin{align*}\cat{Sm^{\circ}}:=&\{(\mathcal G,\beta) \text{ with } \c{G} \to S\text{ a smooth quasi-separated group algebraic space}, \\& \text{ with fiberwise-connected component of identity }\c{G}^0 \text{ separated over }S, \\& \beta\colon \c{G}\to \sLPic^0_{X/S} \text{ a homomorphism such that } \beta_U\colon \c{G}_U\to \Pic^0_{X_U/U} \\& \text{ is an isomorphism}\}.
 \end{align*}
 
The fiberwise connected-component of identity $\c{G}^0$ is an open subgroup space of $\c{G}$ containing the identity section and whose geometric fibers $\c{G}^0_s$ are the connected component of identity of $\c{G}_s$. See \ref{appendix:pi_0} for details on $\c{G}^0$.
\begin{lemma}\label{lemma:G^0Pic^0}Let $(\ca G,\beta)$ be in $\cat{Sm^{\circ}}$. Then $\ca G^0$ is naturally identified with $\Pic^0_{X/S}$. 
\end{lemma}
\begin{proof}
Over $U$ there is already a natural identification $\ca G_U=\ca G^0_U=\Pic^0_{X_U}$. For every point $s$ of $S$ of codimension $1$, the restriction of $\Pic^0_{X/S}$ to $\ca O_{S,s}$ is the identity component of its own N\'eron model.  By \cite[7.4.3]{Bosch1990Neron-models}, the same holds for $\ca G^0$. Now by \cite[XI, 1.15]{Raynaud1970Faisceaux-ample}, the isomorphism $\ca G^0_U\to \Pic^0_{X_U}$ extends uniquely to an isomorphism $\ca G^0\to \Pic^0_{X/S}$.
\end{proof}
Because of the N\'eron mapping property of $\sLPic^0_{X/S}$ (\ref{coro:NMP_strict_log}) there is a natural equivalence between $\cat{{Sm}^{\circ}}$ and the category with objects
\begin{align*}&\{(\mathcal G,\phi) \text{ with } \c{G} \to S\text{ a smooth quasi-separated group-algebraic space}, \\& \text{ with fiberwise-connected component of identity }\c{G}^0 \text{ separated over }S, \\& \phi\colon \c{G}_U\to \Pic^0_{X_U/U} \text{ an isomorphism}\},
 \end{align*} 
so we will not distinguish between the two.

We obtain by restriction of the functor $F$ of \ref{eqn:functor} a functor
$$F^{\circ}\colon \cat{Et^{\circ}}\to \cat{Sm^{\circ}}.$$ 

We are going to construct a quasi-inverse to $F^{\circ}$. We denote by $\pi_0(\ca G)$ the \'etale algebraic space $\c{G}/\c{G}^0$ of \ref{def:componentgroup}. By the universal property of $\c{G}/\c{G}^0$ for maps towards \'etale spaces (\ref{lemma:univ_property_G/G0}), together with the fact that $\pi_0(\sLPic^0_{X/S})=\sTPic^0_{X/S}$ (\ref{G^0_for_slpic}), we obtain a functor
\begin{align*}
\Pi_0\colon \cat{Sm^{\circ}} &\xrightarrow{\hspace{3cm}} \cat{Et^{\circ}} \\
\left(\c{G},\beta\colon \c{G}\to \sLPic^0_{X/S}\right) &\mapsto \left( \pi_0(\c{G}), \pi_0(\beta)\colon \pi_0(\c{G})\to \sTPic^0_{X/S}\right)
\end{align*}

\begin{lemma}\label{lemma:equivalence_quasi-finite}
The functor $F^{\circ}\colon\cat{Et^{\circ}}\to\cat{Sm^{\circ}}$ is an equivalence with $\Pi_0$ as a quasi-inverse.
\end{lemma}

\begin{proof}
Let $\Psi\to \sTPic^0_{X/S}$ be in $\cat{Et^{\circ}}$, with image $\c{G}\to \sLPic^0_{X/S}$ via $F^{\circ}$. The surjective map $\mathcal G\to \Psi$ factors by the universal property via an \'etale surjective map $\pi_0(\mathcal G)\to \Psi$. It remains to show that its kernel $K$ vanishes. We obtain a commutative diagram of exact sequences
\begin{center}
\begin{tikzcd}
0\ar[r] & \c{G}^0 \ar[r]\ar[d] & \c{G} \ar[r,]\ar[d, "="] & \pi_0(\ca G) \ar[r]\ar[d] & 0\\
0\ar[r] & \Pic^0_{X/S} \ar[r] & \c{G} \ar[r] & \Psi \ar[r] & 0
\end{tikzcd}
\end{center}

By \ref{lemma:G^0Pic^0}, the left vertical map is an isomorphism and we conclude.

Conversely, let $\c{G}\in \cat{Sm^{\circ}}$. Because the functor $F^{\circ}$ is defined as a fiber product there is a natural map 
$f\colon \c{G}\to F^{\circ}(\Pi_0(\c{G}))=\pi_0(\c{G})\times_{\sTPic^0} \sLPic^0$, and we obtain a commutative diagram of exact sequences
\begin{center}
\begin{tikzcd}
0\ar[r] & \ca G^0 \ar[r]\ar[d, "h"] & \mathcal G \ar[r, "p\circ f"]\ar[d, "f"] & \pi_0(\mathcal G) \ar[r]\ar[d] & 0\\
0\ar[r] & \Pic^0_{X/S} \ar[r] & \pi_0(\c{G})\times_{\sTPic^0} \sLPic^0\ar[r, "p"] & \pi_0(\mathcal G) \ar[r] & 0
\end{tikzcd}
\end{center}
where the rightmost vertical map is the identity and the leftmost vertical map $h$ is the induced map $\ker(p\circ f)\to \ker(p)$. Since $f$ restricts over $U$ to the identity of $\Pic^0_{X_U}$, so does $h$. It follows that $h$ is the isomorphism of \ref{lemma:G^0Pic^0}, and that $f$ is an isomorphism as well.
\end{proof}

As a corollary of \ref{prop:separated_model} we refine the equivalence $F^{\circ}$.
\begin{definition} We let $\cat{Et^{qf,mono}}$ to be the full subcategory of $\cat{Et^{\circ}}$ of those $(\Psi,\alpha)$ with $\alpha$ a monomorphism (i.e. an open immersion) and $\Psi/S$ quasi-finite. We let $\cat{Sm^{qc,sep}}$ be the full subcategory of $\cat{Sm^{\circ}}$ of those $(\c{G},\phi)$ with $\c{G}\to S$ separated and quasi-compact. 
\end{definition}

Both $\cat{Et^{qf,mono}}$ and $\cat{Sm^{qc,sep}}$ are equivalent to partially ordered sets. For $\cat{Et^{qf,mono}}$ this is clear, and for $\cat{Sm^{qc,sep}}$ we observe that, for an object $(\c{G},\phi\colon \c{G}\to \sLPic^0_{X/S})$ of $\cat{Sm^{qc,sep}}$, $\phi$ is the base change of $\pi_0(\mathcal G)\to \sTPic^0_{X/S}$, by \ref{lemma:equivalence_quasi-finite}. The latter is an open immersion by \ref{prop:separated_model}, so $\phi$ is an open immersion. 


The following corollary allows us to describe all possible smooth separated group models of $\Pic^0_{X_U}$ in terms of open subgroups of the strict tropical Jacobian.
\begin{corollary}\label{coro:equivalence_sep_qf}
The equivalence $F^{\circ}\colon \cat{Et^{\circ}}\to \cat{Sm^{\circ}}$ restricts to an order-preserving bijection
$$F^{*}\colon \cat{Et^{qf,mono}} \to \cat{Sm^{qc,sep}}. $$
\end{corollary}
\begin{proof}
This is immediate by \ref{prop:separated_model}.
\end{proof}



The partially ordered set $\cat{Et^{qf,mono}}$ has a maximal element, namely the quasi-finite \'etale group space $\sTPic^{tor}_{X/S}$ representing the torsion part of the sheaf $\sTPic^0_{X/S}$. From \ref{coro:equivalence_sep_qf} we deduce: 


\begin{theorem}\label{coro:maximal_model}
Let $X/S$ be a log curve over a log regular base $S$, and $U\subset S$ the open where the log structure is trivial.
The partially ordered set of smooth separated group-$S$-models of finite type of $\Pic^0_{X_U}$ has $F^{*}(\sTPic^{tor}_{X/S})=\sPic^{sat}_{X/S}$ as maximum element. Namely, any other such model has a unique open immersion to $\sPic^{sat}_{X/S}$. 
\end{theorem}

\subsection{Possible extensions to the case of log abelian varieties}

It is natural to ask which of the results of this paper remain valid when the Jacobian of a curve is replaced by an arbitrary abelian variety. Suppose that we have a log regular log scheme $S$, and a log abelian variety $A^{\textup{log}}/S$ (which is necessarily an abelian variety $A_U$ over the open locus $U \subseteq S$ on which the log structure is trivial). Now $A^{\textup{log}}$ is a sheaf on $(\textbf{LSch}/S)_{\textup{\'et}}$ which has a tropicalization $A^{\textup{trop}}$ over $S$. We can restrict $A^\textup{log}$ and $A^{\textup{trop}}$ to sheaves on the strict \'etale site $(\textbf{Sch}/\underline{S})_{\textup{\'et}}$ to obtain algebraic spaces
$$
sA^{\textup{log}}, sA^{\textup{trop}}. 
$$
The tropicalization $sA^{\textup{trop}}$ can equivalently be defined as the quotient 
$$
sA^{\textup{trop}} = sA^{\textup{log}}/sA^0
$$
of $sA^{\textup{log}}$ by the (semi-abelian) fiberwise connected component of the identity $sA^0$, and, as described in detail in \cite[4.1.2]{Kajiwara2008Logarithmic-abe} also has an explicit combinatorial description \'etale locally in $S$ as
$$
\textup{Hom}(X,\overline{M}_S^{\gp})_{(Y)}/Y
$$
for lattices $X$ and $Y$. Here the subscript $(Y)$ indicates a subgroup of $\textup{Hom}(X,\overline{M}_S^{\gp})$ \cite[3.1]{Kajiwara2008Logarithmic-abe}, analogous to the bounded monodromy subgroup of the Jacobian. 

\begin{conjecture}\label{conj:logNMPforlogAV}
The log abelian variety $A^{\textup{log}}$ has the log N\'eron mapping property (\ref{def:log_NMP}) with respect to $U$.
\end{conjecture}

\Cref{conj:logNMPforlogAV} would in particular imply that $sA^{\textup{log}}$ is always a N\'eron model for $A_U$, but is rarely separated: it is separated if and only if $sA^{\textup{trop}}$ is finite. At the moment, we do not have a proof of \ref{conj:logNMPforlogAV}. Our proof for the Jacobian uses the geometry of the curve to produce the extension. On the other hand, the proof that abelian varieties are their own N\'eron models goes by extending line bundles on the dual abelian variety. This argument would extend to the case of logarithmic abelian varieties if we had a theory of log Picard functors for higher dimensional logarithmic schemes which satisfies analogues of \ref{property: line_bundle_representative},  \ref{property:logpic_modifications}, \ref{property:log_pic_is_a_log_stack}, and the usual duality axioms. 

Conditional on \ref{conj:logNMPforlogAV}, our proof of \ref{thm:intro_bijection} goes through verbatim to show that there is a bijection between quasi-finite open subgroups of $sA^{\textup{trop}}$ and smooth, separated, quasi-compact $S$-group models of $A_U$.

\section{Alignment and separatedness of strict log Pic}\label{sec:alignment}

For $X/S$ a log curve over a log regular base $S$ with $U\subset S$ the largest open where the log structure is trivial, we have shown in \ref{coro:NMP_strict_log} that $\sLPic^0_{X/S}$ is the N\'eron model of $\Pic^0_{X_U/U}$. 
It is worth stressing the fact that classically, the term N\'eron model is reserved for separated, quasi-compact models satisfying the N\'eron mapping property. The strict logarithmic Jacobian fails in general to satisfy these properties, as observed in \ref{example:1-gon}.

In the papers \cite{Holmes2014Neron-models-an}, \cite{Orecchia2018A-criterion-for}, \cite{Poiret}, several criteria were introduced for the Jacobian of a prestable curve $X/S$ (or for an abelian variety in \cite{Orecchia2019A-monodromy-cri}) to admit a separated, quasi-compact N\'eron model. They are all closely related to the general notion of \emph{log alignment} that we introduce here:



\begin{definition}
Let $\o{M}$ be a sharp fs monoid. We call the $1$-dimensional faces of $\o{M}\otimes_{\bb Z}\bb R_{\geq 0}$ the \emph{extreme rays} of $\oM$. 
\end{definition}

\begin{definition}
A \emph{cycle} in a graph is a path that begins and ends at the same vertex, and which otherwise repeats no vertices. A subset $S$ of the edges of a graph is called \emph{cycle-connected} if for every pair $e$, $e' \in S$ of distinct edges there exists a cycle in $S$ containing $e$ and $e'$. It is shown in \cite[lemma 7.2]{Holmes2014A-Neron-model-o} that the maximal cycle-connected subsets (which are there called circuit-connected) form a partition of the edges of the graph. 
\end{definition}

\begin{definition}
We say that a tropical curve $\frak X$ metrized by a sharp fs monoid $\oM$ is \emph{log aligned} when for every cycle $\gamma$ in $\frak X$, all lengths of edges of $\gamma$ lie on the same extreme ray of $\oM$. Let $X \rightarrow S$ be a log curve. We say that $X/S$ is \emph{log aligned} at a geometric point $\bar s$ of $\ul S$ when the tropicalization of $X$ at $\bar s$ is log aligned. We say that $X/S$ is \emph{log aligned} if it is log aligned at every geometric point of $\ul S$.
\end{definition}

\begin{lemma}\label{lem:log_alignment_preserved_under_subdivisions}
Let $\oM$ be a sharp fs monoid, $\frak X$ a tropical curve metrized by $\oM$ and $\frak Y$ a subdivision of $\frak X$. Then $\frak Y$ is log aligned if and only if $\frak X$ is.
\end{lemma}

\begin{proof}
It suffices to treat the case where $\frak Y$ is a basic subdivision of $\frak X$. Suppose it is, and call $e$ the subdivided edge: it is replaced in $\frak Y$ by a chain of two edges $e_1, e_2$ of the same total length. There is a canonical bijection between cycles of $\frak X$ and of $\frak Y$. Let $\gamma$ be a cycle of $\frak X$. It suffices to show that the lengths of edges of $\gamma$ in $\mathfrak X$ lie on the same extreme ray of $\oM$ if and only if the lengths of edges of $\gamma$ in $\mathfrak Y$ do. If $\gamma$ does not contain $e$, this is clear. Otherwise, it follows from observing that the length of $e$ is in an extreme ray $R$ of $\oM$ if and only if the lengths of $e_1$ and $e_2$ are both in $R$.
\end{proof}

\begin{theorem}\label{thm:alignment_conditions}
Let $X/S$ be a log curve. Consider the following conditions:
\begin{enumerate}
\item $X/S$ is log aligned;
\item $\strTroPic^0_{X/S}$ is quasi-finite over $S$;
\item $\strLogPic^0_{X/S}$ is separated over $S$.
\end{enumerate}
Then, we have $({\rm i}) \iff ({\rm ii}) \implies ({\rm iii})$. If $S$ is log regular and $U\subset S$ is the largest open where the log structure is trivial, we additionally have $(iii) \implies (ii)$, and the conditions above are equivalent to the following two:
\begin{enumerate}[resume]
\item $\sLPic^0_{X/S}$ is a separated N\'eron model of finite type for $\Pic^0_{X_U}$.
\item $\Pic^0_{X_U/U}$ admits a separated N\'eron model of finite type over $S$.
\end{enumerate}
 \end{theorem}

\begin{proof}
First, \ref{prop:separated_model} for $(\Psi,\alpha)=(\sTPic^0_{X/S},\on{Id})$ gives $({\rm ii}) \implies ({\rm iii})$ and if $S$ is log regular also $({\rm iii})\implies ({\rm ii})$. If $\sLPic^0_{X/S}$ is separated then by \ref{prop:separated_model} it is also quasi-compact, hence of finite type. The equivalence of (iii), (iv) and (v) in the log regular case then follows from \ref{thm:Neron} and the uniqueness of N\'eron models.

It remains to prove $({\rm i}) \iff ({\rm ii})$. By \ref{coro:maximal_model}, $\sTPic^{tor}_{X/S}$ is the maximum open quasi-finite subgroup of $\sTPic^0_{X/S}$. Condition (ii) is then equivalent to $\sTPic^{tor}_{X/S}=\sTPic^0_{X/S}$, which in turn is equivalent to $\sTPic^0_{X/S}$ having finite fibers.
We immediately reduce to the case where $S$ is a geometric log point, and we write $M \coloneqq M_S(S)$. Denote by $\mathfrak{X}$ the tropicalization of $X/S$, and by $\mathfrak{X}_1,...,\mathfrak{X}_n$ the maximal cycle-connected components of $\mathfrak{X}$. We have a canonical isomorphism 
\begin{align*}
H_1(\mathfrak{X}) = \bigoplus_{i=1}^n H_1(\mathfrak{X_i}).
\end{align*}
Suppose first that $\mathfrak{X}/\o{M}$ is log aligned, so that for every $0\leq i\leq n$ there is an irreducible element $\rho_i$ of an extremal ray of $\o M$ such that all edges of $\mathfrak{X}_i$ have length in $\NN\rho_i$. Thus $\mathfrak{X}_i$ can be seen as a tropical curve metrized by $\NN \rho_i$, and any bounded monodromy map $H_1(\mathfrak{X}_i) \to \o M^{\gp}$ factors uniquely through the inclusion $\ZZ\rho_i \to \o M^{\gp}$. We get isomorphisms
\begin{align*}
\Hom(H_1(\mathfrak{X}),\o{M}^{\gp})^\dagger = \bigoplus_{i=1}^n \Hom(H_1(\mathfrak{X_i}),\o{M}^{\gp})^\dagger = \bigoplus_{i=1}^n \Hom(H_1(\mathfrak{X_i}),\ZZ \rho_i)^\dagger,
\end{align*}
where the first equality holds since bounded monodromy can be checked separately on each $\mathfrak{X}_i$ by \ref{rem:length_additive_if_disjoint_supports}. Quotienting by $H_1(\mathfrak{X})$, we obtain
\begin{align*}
\Tropic^0(\mathfrak{X}/\o{M}) = \bigoplus_{i=1}^n \Tropic^0(\mathfrak{X}_i/\o M) = \bigoplus_{i=1}^n \Tropic^0(\mathfrak{X}_i/\NN \rho_i).
\end{align*}
The right hand side is finite, as the rank of $\Hom(H_1(\mathfrak{X}_i),\ZZ \rho_i)$ is equal to the rank of $H_1(\mathfrak{X}_i)$. 

For the reverse implication, suppose $X/S$ is not log aligned; we will show $$\Tropic^0(\mathfrak{X}/\o{M}) = \Hom(H_1(\mathfrak{X}),\o{M}^\gp)^\dagger/H_1(\mathfrak{X})$$ is not finite, by showing its rank is at least $1$. Note that if $\o{M} \rightarrow \o{M}'$ is a finite index homomorphism, a homomorphism $\phi: H_1(\mathfrak{X}) \rightarrow \o{M}^\gp$ has bounded monodromy if and only if its composition with $\o{M}^\gp \rightarrow \o{M}'^\gp$ has bounded monodromy. Thus, the rank of $\Tropic^0(\mathfrak{X}/\o{M})$ is equal to the rank of $\Tropic^0(\mathfrak{X},\o{M}')$ for any finite index inclusion $\o{M} \rightarrow \o{M}'$. Let $\ell(e) \in \o{M}$ denote the length of the edge $e$ in $\mathfrak{X}$. As the extreme rays of $\o{M}$ span $\o{M}$ over $\QQ$, and we are free to replace $\o{M}$ by finite index extensions, we can assume that each length $\ell(e)$ can be written as a sum of elements in $\o{M}$ that lie on the extreme rays of $\o{M}$. We may then subdivide $\mathfrak{X}$ so that each edge in the subdivision has length along the extreme rays of $\o{M}$. Using \ref{lem:log_alignment_preserved_under_subdivisions} and the invariance of the tropical Jacobian under subdivisions of $\mathfrak{X}$, we may then assume that each edge of $\mathfrak{X}$ has length which lies along an extreme ray of $\o{M}$.  Pick a spanning tree $T$ of $\mathfrak{X}$. The edges $e_1,...,e_r$ not in $T$ correspond to cycles $\gamma_1,...,\gamma_r$ forming a basis of $H_1(\mathfrak{X})$. By hypothesis, $\mathfrak{X}$ is not log aligned, so one of the $\gamma_i$, for example $\gamma_1$, has length not belonging to an extremal ray of $\o M^\gp$. Therefore, there exists an edge $e$ in $\gamma_1$ of length along an extreme ray of $\o{M}$ different than the ray containing the length of $e_1$. 

We claim that the intersection pairings of the family $(e,e_1,...,e_r)$ are independent bounded monodromy maps $H_1(\mathfrak{X}) \to \o M^\gp$. The fact that intersection pairing with an edge has bounded monodromy is general: for any edge $e$, and any cycle $\gamma \in H_1(\mathfrak{X})$, the intersection pairing $e.\gamma$ evidently has length bounded by the length of $\gamma$. To see that the pairings are independent, notice that $e_i.\gamma_j$ is $\delta_{ij} \ell(e_i)$ where $\delta_{ij}$ is the Kronecker delta. Consider a linear combination $b=ae+\sum a_ie_i$ with coefficients in $\ZZ$, and suppose the intersection pairing of $b$ is trivial. Then $a\ell(e) + a_1\ell(e_1) = b.\gamma_1=0$, combined with the fact $e$ and $e_1$ have independent lengths, yields $a=a_1=0$. Hence for $j>1$ we have $a_j\ell(e_j)=b.\gamma_j=0$, from which we deduce $a_j=0$.

From this, we obtain
\begin{align*}
\on{rank}(\Hom(H_1(\mathfrak{X}),\o M^{\gp})^\dagger) \geq r+1 > r = \on{rank}(H_1(\mathfrak{X})).
\end{align*}
Thus, $\Tropic^0(\mathfrak{X}/\o{M}) = \Hom(H_1(\mathfrak{X}),\o{M}^\gp)^\dagger/H_1(\mathfrak{X})$ is not finite.
\end{proof}

\section{The strict logarithmic Jacobian and the Picard space}\label{sec:etale_closure}

Over  Dedekind base $S$, Raynaud constructed the N\'eron model of the Jacobian of a curve $X/S$ as the quotient of the relative Picard space by the closure $\bar e$ of the unit section (see \cite[Section 9.5]{Bosch1990Neron-models}). When $\dim S >1$ the closure $\bar e$ is in general neither $S$-flat nor a subgroup, and so this quotient is not representable. In \cite{Holmes2014Neron-models-an} and \cite{Orecchia2018A-criterion-for}, necessary and sufficient conditions for the flatness of $\bar e$ are given, proving the existence of separated N\'eron models when these conditions hold. In this section we show that Raynaud's approach can be extended over higher-dimensional bases even when $\bar e$ is not $S$-flat, simply by replacing $\bar e$ by its largest open subspace $\bar e^\et$ which is \'etale over $S$. This allows us to describe the N\'eron model (constructed above as the algebraic space $\sLPic^0$) as the quotient of the Picard space by $\bar e^\et$, under somewhat more restrictive assumptions on $X/S$. The construction of the N\'eron model in \cite{Poiret} is done by gluing local models of this form.


Letting $\pi\colon X\to S$ be a log curve, we write $\Pic^{tot 0}$ for the kernel of the composition $\Pic_{X/S}\to \ZZ[\Irr_{X/S}]\xrightarrow{\Sigma} \ZZ$. There is a canonical map
\[
\Pic^{tot 0}_{X/S} \to \sLPic^0_{X/S}
\]
taking a line bundle to the associated log line bundle. 

Write $\bar e$ for the schematic closure of the unit section in $\Pic^{tot 0}_{X/S}$. If $U$, $V \hra \bar e$ are open immersions \'etale over $S$, then the same is true of their union. Hence $\bar e$ has a largest open subscheme which is \'etale over $S$, which we denote by $\bar e^{\et}$, a locally closed subscheme of $\Pic^{tot 0}_{X/S}$. 

\begin{theorem}\label{thm:sLPic_as_quotient}
Suppose that $S$ is log regular. Then
\begin{enumerate}
\item The map 
$f \colon \Pic^{tot 0}_{X/S} \to \strLogPic^0_{X/S}$ has kernel $\bar e^{\et}$; 
\item If in addition  $\ul X$ is regular, then $f$ is surjective.
\end{enumerate}
\end{theorem}

\begin{remark}
A log modification $X'\to X$ induces an isomorphism $\sLPic^0_{X/S}\to \sLPic^0_{X'/S}$by \ref{property:logpic_modifications}. On the other hand $\Pic^{tot 0}_{X/S}\to \Pic^{tot 0}_{X'/S}$ is an open immersion but in general not an isomorphism. When $\ul S$ is log regular and regular, we can always find, \'etale locally on $S$, a log modification of $X/S$ with regular total space. To see this, note that since $S$ is locally Noetherian, it has an \'etale cover by nuclear schemes by \ref{lemma:cover_by_nuclear}, so we can assume that $S$ is nuclear. Let $\mathfrak{X}$ denote the tropicalization of $X$ over the closed stratum. As $S$ is log regular and regular, every edge of $\mathfrak X$ is marked by an element of a free monoid $\NN^r$. Put $\mathfrak X'=\mathfrak X$. As long as $\mathfrak X'$ has an edge whose length is not one of the generators of $\NN^r$, replace $\mathfrak X'$ by any basic subdivision at that edge. The process terminates, and provides a maximal subdivision of $\mathfrak X$. This subdivision lifts to a  log modification $X' \rightarrow X$ whose total space is regular: $X'$ is evidently log regular, and the log structure around a node of length a generator of $\NN^r$ is isomorphic to $\NN^r \oplus_\NN \NN^2 \cong \NN^{r+1}$. Combining this observation with \ref{thm:sLPic_as_quotient}, we obtain a local description of $\strLogPic^0_{X/S}$ as a quotient of a Picard space.
\end{remark}
\begin{corollary}Let $X/S$ be a prestable curve over a toroidal variety, smooth exactly over the complement of the toroidal boundary. Assume $X$ is regular. Then the quotient $\Pic^{tot 0}_{X/S}/\bar e^\et$ is the N\'eron model of the Jacobian of $X$. 
\end{corollary}

\begin{proof}[Proof of \ref{thm:sLPic_as_quotient}]\leavevmode \\%
\textbf{Proof of (i). }
 We write $\Psi$ for the kernel of the summation map $\ZV\to \ZZ$; this is \'etale as it is the kernel of a map of \'etale group spaces. The multidegree map $\Pic^{tot 0}_{X/S}\to \Psi$ is the cokernel of the open immersion $\Pic^0_{X/S}\to \Pic^{tot 0}_{X/S}$.
We obtain a commutative diagram with exact rows
\[
\begin{tikzcd}
0 \ar[r] & \Pic^0_{X/S} \ar[equal]{d}\ar[r] & \Pic^{tot 0}_{X/S} \ar[r]\ar[d, "f"] & \Psi \ar[d, "g"] \ar[r] & 0\\
0 \ar[r] & \Pic^0_{X/S} \ar[r] & \strLogPic^0_{X/S} \ar[r] & \strTroPic^0_{X/S} \ar[r] & 0
\end{tikzcd}
\]
where $K:=\ker f$ is equal to $\ker g$ by the snake lemma. The map $g$ is a map of \'etale group spaces (\ref{thm:sPic_representable}), hence its kernel is \'etale, i.e. $K$ is \'etale. 

Note that the locus $U\hra S$ over which the morphism $X \rightarrow S$ is smooth is schematically dense by our log regularity assumption. One checks immediately that $K$ is trivial over $U$, hence the map $K \to \Pic^{tot0}_{X/S}$ factors via the closure $\bar e$ of the unit section (since $K \to S$ is \'etale, so the pullback of $U \sub S$ is schematically dense in $K$). The map $K \to \bar e$ is a quasi-compact immersion. To show that it is open, we may pick a geometric point $p\in K$ and restrict to the strict henselization of $S$ at $p$. There is a unique section $S\to K$ through $p$, which by \ref{lem:open_image} is open in $\bar e$. Thus $K \to \bar e$ is open, and factors through an open immersion
$
K \sub \bar e^{\et}$.

The reverse inclusion is easier: the map $\bar e{^\et}\to \sLPic^0_{X/S}$ is zero when restricted to $U$, so by the N\'eron mapping property of $\sLPic^0_{X/S}$ (\ref{thm:Neron}), the map $\bar e^{\et}\to \sLPic^0_{X/S}$ is zero and therefore $\bar e^{\et}\subset K$.

\noindent \textbf{Proof of (ii). }
Note that since $S$ is log regular, so is $X$; thus, if in addition $\ul{X}$ is regular, its log structure is locally free. Since the log structure of $S$ is isomorphic to the log structure of $X$ away from the singular points of the fibers, the log structure on $S$ must therefore also be locally free, and thus $\ul{S}$ must also be regular. As the N\'eron model $N=\sLPic^0_{X/S}$ of $\Pic^0_{X_U/U}$ is smooth over $S$, it follows that $X_N$ is regular as well. The canonical isomorphism $N_U=\Pic^0_{X_U/U}$ corresponds to a line bundle $L$ on $X_{N_U}$, represented by a Cartier divisor $D$. The scheme-theoretical closure $\o D$ of $D$ in $X_N$ is Cartier by regularity of the latter. Thus the line bundle $\ca O(\o D)$ provides a lift of $L$ under the natural morphism $\Pic^{tot 0}_{X/S} \to N$, which is therefore surjective.
\end{proof}

\begin{lemma}\label{lem:open_image}
Let $S$ be the spectrum of a local ring, $U \hra S$ open, and $X \to S$ a morphism such that $X_U \to U$ is an isomorphism, and $X_U$ is schematically dense in $X$. Let $\sigma\colon S \to X$ be a section. Then $\sigma$ is open. 
\end{lemma}
\begin{proof}
Writing $s$ for the closed point of $S$, let $\sigma(s) \in V \hra X$ be an affine open neighbourhood; then $\sigma$ factors via $V$ (the preimage of $V$ via $\sigma$ is open and contains $s$). Write $\sigma'\colon S \to V$ for the factored map. 

Since $V \to S$ is separated, the map $\sigma'\colon S \to V$ is a closed immersion. On the other hand, its image contains the schematically dense $V_U \hra V$, hence $\sigma'$ is an isomorphism. 
\end{proof}

\appendix

\section{The functor of connected components of a smooth quasi-separated group algebraic space}\label{appendix:pi_0}

Throughout this appendix, $S$ denotes a scheme and $G/S$ a group algebraic space with unit section $e \in G(S)$. We extend some results of Romagny
\cite{Romagny2011Composantes-con} to the case where $G/S$ is smooth and quasi-separated, avoiding the quasi-compactness assumptions of \cite{Romagny2011Composantes-con}. 


\begin{lemma}\label{lemma:fiberwise-connected_cpts_are_open}
Suppose that $G/S$ is quasi-separated, flat, locally of finite presentation, and has reduced geometric fibers. Then there is a unique open subspace $G_0$ of $G$ such that each fiber $G_{s,0}$ of $G_0/S$ is the connected component of $G_s$ containing $e(s)$. Moreover, $G_0$ is a subgroup of $G$. 
\end{lemma}
\begin{definition}\label{def:componentgroup}
We call $G_0$ the \emph{fiberwise-connected component of identity in $G$}. The sheaf quotient $G/G_0$ is a group algebraic space by \cite[Proposition 8.3.9]{Bosch1990Neron-models}, which we call the \emph{group of connected components} of $G$.
\end{definition}
\begin{proof}[Proof of \ref{lemma:fiberwise-connected_cpts_are_open}]
The assertion that $G_0$ is a subgroup of $G$ is immediate from the continuity of the multiplication and inversion operations. The first part of the statement is also immediate, from \cite[Proposition 2.2.1]{Romagny2011Composantes-con}, if in addition $G/S$ is quasi-compact. 

We will prove the general case by reduction to the quasi-compact case. We define a sub\emph{set} $\abs{G_0}$ of the underlying topological space $\abs{G}$ of $G$  to be the union of the $G_{s,0}$ as $s$ runs over $\abs{S}$. It suffices to show that $\abs{G_0}$ is open in $\abs{G}$. Given $s\in S $ and $x \in G_{0,s}$, let $U \to G$ be an \'etale map from an affine scheme with $x$ and $e(s)$ in its topological image $\ca W\subset |G|$; the latter is open and we let $W$ be the corresponding open subspace of $G$. Base-changing to an open neighbourhood of $s$ in $S$, we can assume $e$ factors through $W$. Then by \cite[Proposition 2.2.1]{Romagny2011Composantes-con} we obtain an open subspace $W_0 \to W$ through which $e$ factors and whose fibers over $S$ are connected components of the fibers of $W$.

Since any connected group scheme over a field is irreducible, for any $s\in S$ the connected component of identity $G_{s,0}$ is irreducible. Then the intersection $G_{s,0}\cap W_s$ is connected and therefore coincides with $W_{0,s}$. This shows that $|W_0|=|W|\cap \abs{G_0}$. In particular, we have $x\in |W_0| \subset \abs{G_0}$ with $|W_0|$ open in $|G|$. 
\end{proof}


\begin{lemma}Suppose that $G/S$ is smooth and quasi-separated. Then the structure morphism $G/G_0 \to S$ is \'etale. 
\end{lemma}
\begin{proof}
We prove this locally at $x \in G/G_0$. Since $G/G_0 \to S$ is smooth we can choose $S' \to S$ \'etale and a section $S' \to G/G_0$ through $x$. Translating by this section, we may assume $x$ lies in the image of the unit section $u_0\colon S \to G/G_0$. It then suffices to show that $G/G_0 \to S$ is \'etale in an open neighbourhood of the unit section, but the unit section is itself open (as the image $G_0 / G_0$ of the open $G_0 \hra G$). 
\end{proof}

\begin{lemma}\label{lemma:univ_property_G/G0}
Suppose that $G/S$ is smooth and quasi-separated, and let $T \to S$ be an \'etale algebraic space. Then any $S$-morphism $G \to T$ factors uniquely via $G \to G/G_0$. 
\end{lemma}
\begin{proof}
Fix an $S$-morphism $f\colon G_0 \to T$, and write $e\colon S \to T$ for the map induced by the unit section of $G_0$. By the universal property of the quotient $G/G_0$, it suffices to show that the following diagram commutes.
\begin{equation}
 \begin{tikzcd}
  G_0 \arrow[r, "f"] \arrow[d]& T \\
  S \arrow[ur,"e"] &
\end{tikzcd}
\end{equation}

If $S$ is the spectrum of a separably closed field then $T$ is a scheme (e.g. by \ref{cor:loc_constructible_implies_representable_big}), $e$ is an open and closed immersion, and $G_0$ is connected, so the result is clear. In the general case it follows that the diagram on $k$-points commutes for any separably closed field $k$. By descent, we may replace $T$ by an \'etale cover and assume it is a scheme. Since the diagonal of $T$ is an open immersion, the equalizer of $f$ and $G_0 \to S \to T$ is open in $G$. Since this open subspace contains all geometric points, it is $G_0$.
\end{proof}

\begin{lemma}\label{G^0_for_slpic}
Let $X/S$ be a log curve and $\c{G}=\sLPic^0_{X/S}$. Then $\c{G}_0=\Pic^0_{X/S}$.
\end{lemma}
\begin{proof}
The natural map $\c{G}/\c{G}_0\to \sTPic^0_{X/S}$ coming from the universal property (\ref{lemma:univ_property_G/G0}) induces a  commutative diagram of exact sequences
\begin{center}
\begin{tikzcd}
0\ar[r] & \c{G}_0 \ar[r]\ar[d] & \c{G} \ar[r,]\ar[d, equals] & \c{G}/\c{G}_0 \ar[r]\ar[d] & 0\\
0\ar[r] & \Pic^0_{X/S} \ar[r] & \c{G} \ar[r] & \sTPic^0_{X/S} \ar[r] & 0
\end{tikzcd}
\end{center}
By the snake lemma, the left vertical map is injective, and its cokernel is identified with the kernel $K$ of the right vertical map. As the latter is \'etale, so is $K$. However, $\Pic^0_{X/S}$ has connected geometric fibers, hence the map $\Pic^0_{X/S}\to K$  is constantly zero. 
\end{proof}
\bibliography{prebib}

\acknowl{This paper has benefitted from the generous input of many people. In particular, we would like to thank: Alberto Bellardini and Arne Smeets for discussions of log N\'eron models and log abelian varieties; Luc Illusie for pointing out a problem with a previous version of our proof of representability; Chris Lazda and Arne Smeets for co-organising the intercity seminar on \cite{Molcho2018The-logarithmic} which initiated this collaboration; Laurent Moret-Bailly for giving us a reference for \ref{cor:loc_constructible_implies_representable_big}; Jonathan Wise for many useful conversations on log Jacobians; the DIAMANT cluster for funding a research visit of Sam Molcho to the Netherlands.}

\end{document}